\documentclass{amsart}

\usepackage{graphicx}
\usepackage[latin1]{inputenc}

\newtheorem{thm}{Theorem}[section]

\newtheorem{prop}[thm]{Proposition}

\newtheorem{dfn}[thm]{Definition}
\newtheorem{rem}[thm]{Remark}
\numberwithin{equation}{section}

\begin{document}

\title[]{Artin's braids, Braids for three space, and groups $\Gamma_{n}^{4}$ and $G_{n}^{k}$}

\author{S.Kim}

\address{Department of Fundamental Sciences, Bauman Moscow State Technical University, Moscow, Russia \\
ksj19891120@gmail.com}

\author{V.O.Manturov}

\address{Chelyabinsk State University and Bauman Moscow State Technical University, Moscow, Russia \\
vomanturov@yandex.ru}

\begin{abstract}
We construct a group $\Gamma_{n}^{4}$ corresponding to the motion of points in $\mathbb{R}^{3}$ from the point of view of Delaunay triangulations. We study homomorphisms from pure braids on $n$ strands to the product of copies of $\Gamma_{n}^{4}$.  We will also study the group of pure braids in $\mathbb{R}^{3}$, which is described by a fundamental group of the restricted configuration space of $\mathbb{R}^{3}$, and define the group homomorphism from the group of pure braids in $\mathbb{R}^{3}$ to $\Gamma_{n}^{4}$. In the end of this paper we give some comments about relations between the restricted configuration space of $\mathbb{R}^{3}$ and triangulations of the 3-dimensional ball and Pachner moves.
\end{abstract}

\maketitle

\section{Introduction}
In~\cite{Manturov} the second named author defined a family of groups $G_{n}^{k}$ for two positive integers $n>k$, and formulated the following principle:

{\it If dynamical systems describing a motion of n particles, admit some good codimension one property governed by exactly $k$ particles, then these dynamical system have a topological invariant valued in $G_{n}^{k}$. }

The main examples coming from $G_{n}^{k}$-theory are homeomorphisms from the $n$-strand pure braid group to the groups $G_{n}^{3}$ and $G_{n}^{4}$ \cite{ManturovNikonov}. If we consider a motion of $n$ pairwise distinct points on the plane and choose the property ``some three points are collinear'', then we get a homomorphism from the pure $n$-strand braid group $PB_{n}$ to the group $G_{n}^{3}$. If we choose the property ``some four points belong to the same circle or line'', we shall get a group homomorphism from $PB_{n}$ to $G_{n}^{4}$. 

In other words, in our examples we look for ``walls'' in the configuration space $C_{n}(\mathbb{R}^{2})$ where some three points are collinear (or some four points are on the same circle or line). This condition can be well defined for $C_{n}(\mathbb{R}P^{2})$.

But what is the ``good'' codimension one property if we try to study similar configuration spaces or braids for some other topological spaces? First, our conditions will heavily depend on the metrics: the property ``three points are collinear'' is metrical. On the other hand, even having some good metrics chosen, we meet other obstacles because we need to know what is a ``line''. For example, there is no unique geodesics passing through two points in the general case. And when finding all possible geodesics and trying to write a word corresponding to it, we shall see that ``a word will contain infinitely many letters'' in the case of irrational cable.

The detour for this problem will be as follows: we shall consider the condition ``locally'' and instead of ``global configurations of spaces''. We shall consider only Vorono\"{i} tiling or Delaunay triangulations. From this point of view, we deal with the space of triangulations with a fixed number of triangles, where any two adjacent triangulations are related by a Pachner move \cite{Nabutovsky}, which is closely related to the group $\Gamma_{n}^{4}$ (see Definition~\ref{dfn_Gamma}). 

On the other hand, the triangulation of spaces and Pachner moves are also related to Yang-Baxter maps (see \cite{Dynnikov}). Moreover, in~\cite{ChoZichertYun, HikamiInoue} a boundary-parabolic $PSL(2,\mathbb{C})$-representation of $\pi_{1}(S^{3} \backslash K)$ for a hyperbolic knot $K$ is studied by using cluster algebras and {\it flips} -- Pachner moves for 2-dimensional triangulations. Since the group $\Gamma_{n}^{4}$ is closely related to triangulation of spaces and Pachner moves, it can be expected to obtain invariants by means of the group $\Gamma_{n}^{4}$ not only for braids, but also for knots, which are obtained by closing braids.

Now we consider restricted spaces $C_{n}'(\mathbb{R}^{k-1})$ defined as follows: a point in $C_{n}'(\mathbb{R}^{k-1})$ is a set of $n$ distinct points in $\mathbb{R}^{k-1}$, where every $(k-1)$ points are in general position. In particular, for $k=3$, the only condition is that no two points among the given $n$ points coincide and the fundamental group $\pi_{1}( C_{n}'(\mathbb{R}^{3-1}))$ is precisely the Artin pure braid group. For $k=4$, for points $x_{1},\cdots, x_{n}$ in three-space we require that no three points are collinear (though some four points can belong to the same plane). We call elements in $\pi_{1}( C_{n}'(\mathbb{R}^{4-1}))$ {\it braid on $n$ strand for $\mathbb{R}^{3}$}.
We call elements in $\pi_{1}( C_{n}'(\mathbb{R}P^{4-1}))$ {\it braid on $n$ strand for $\mathbb{R}P^{3}$}. In~\cite{Manturov_Gnk_config} the following statement is proved:

\begin{prop}
There exists the group homomorphism $f_{n}^{k}$ from $\pi_{1}( C_{n}'(\mathbb{R}P^{k-1}))$ to $G_{n}^{k}$.
\end{prop}
Roughly speaking, for a path $\gamma \in \pi_{1}( C_{n}'(\mathbb{R}P^{k-1}))$ the mapping $f_{n}^{k}(\gamma)$ is defined by writing $a_{m}$ when exactly one $k$-tuple of points belongs to a $(k-2)-$plane, where $m$ is the set of indices for $k$ points on the $(k-2)-$plane.

The paper is organized as follows. In Section 2, we introduce basic definitions and draw pictures describing the motion of points in $\mathbb{R}^{2}$ or in $\mathbb{R}^{3}$. In Section 3, we define the homomorphism from $PB_{n}$ to $\Gamma_{n}^{4}$. In Section 4 we shall construct a homomorphism $\psi_{n}$ from $PB_{n}$ to $\Gamma_{n}^{4} \times \Gamma_{n}^{4}$, which is defined as the homomorphism from $PB_{n}$ to $G_{n}^{4}$ in \cite{ManturovNikonov}, but separating four points on the circle with respect to the number of points inside the circle modulo $2$. In Section 5 we will construct a homomorphism $\pi_{n}$ from $G_{n}^{4}$ to $\Gamma_{n}^{4} \times \Gamma_{n}^{4}$ and show that the homomorphism $\psi_{n}$ can be presented by the composition of the homomorphism from $PB_{n}$ to $G_{n}^{4}$ in \cite{ManturovNikonov} and the homomorphism $\pi_{n}$.

\section{Pictures and basic definitions}

Let us use the notation $\bar{n} :=\{1, \cdots n\}$.
Following \cite{Lee}, we choose the presentation for the pure Artin braid group.
\begin{dfn}
The pure braid group $PB_{n}$ of $n$ strands is the group given by group presentation generated by $\{ b_{ij} | i,j \in \bar{n}, i<j\}$ subject to the following relations:
\begin{enumerate}
\item $b_{ij}b_{kl} = b_{kl}b_{ij}$ for $i,j,k,l \in \bar{n} $ such that $i<j<k<l$ or $i<k<l<j$; 
\item $b_{ij}b_{ik}b_{jk} = b_{ik}b_{jk}b_{ij} = b_{jk}b_{ij}b_{ik}$ for $i,j,k \in \bar{n} $ such that $i<j<k$;
\item  $b_{ik}b_{jk}b_{jl}b_{jk} = b_{jk}b_{jl}b_{jk}b_{ik} $ for $i,j,k,l \in \bar{n} $ such that $i<j<k<l$.
\end{enumerate}

\end{dfn}

\begin{dfn}
The group $G_{n}^{4}$ is the group given by group presentation generated by $\{ a_{\{ijkl\}}~|~ \{i,j,k,l\} \subset \bar{n}, |\{i,j,k\}| = 4\}$ subject to the following relations:

\begin{enumerate}
\item $a_{\{ijkl\}}^{2} = 1$ for $\{i,j,k,l\} \subset \bar{n}$, 
\item $a_{\{ijkl\}}a_{\{stuv\}} = a_{\{stuv\}}a_{\{ijkl\}}$ for $| \{i,j,k,l\} \cap \{s,t,u,v\} | < 3$,
\item $(a_{\{ijkl\}}a_{\{ijkm\}}a_{\{ijlm\}}a_{\{iklm\}}a_{\{jklm\}})^{2} = 1$ for distinct $i,j,k,l,m$.
\end{enumerate}
We use the notation $ a_{ijkl} := a_{\{ijkl\}}$.

\end{dfn}

Now we recall the group homomorphism from $PB_{n}$ to $G_{n}^{4}$, which is defined in \cite{ManturovNikonov}.
Pure braids can be considered as dynamical systems whose initial and final states coincide. 

\begin{figure}[h!]
 \centering
 \includegraphics[width = 6cm]{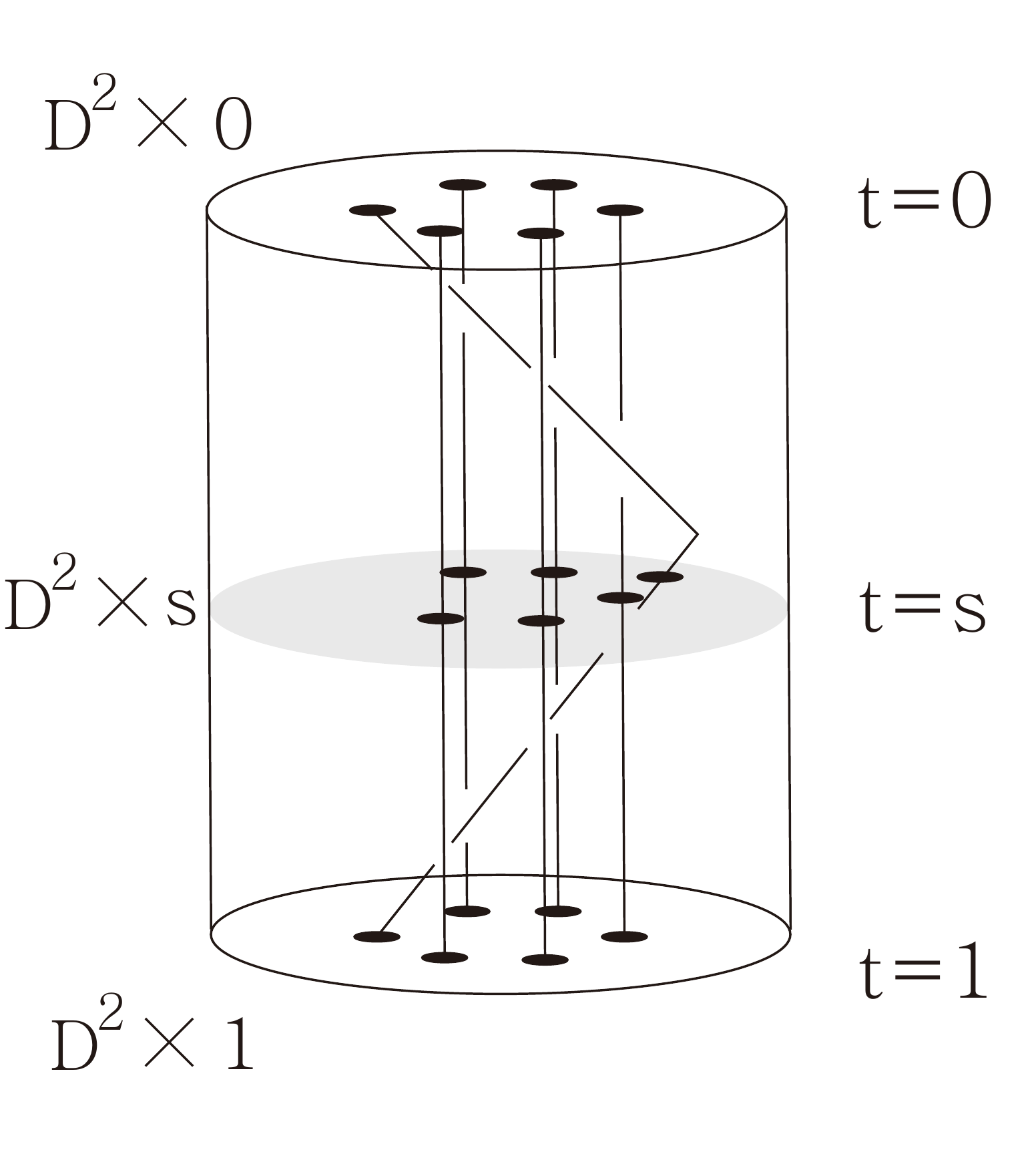}
 \caption{Dynamical system corresponding to $b_{ij}$}\label{exa-dyna}
\end{figure}

Let $\Gamma = \{(t, t^{2}) | t \in \mathbb{R} \} \subset \mathbb{R}^{2}$ be the graph of the function $y =x^{2}$. Consider a rapidly increasing sequence of positive numbers $t_{1}, t_{2}, \cdots, t_{n}$ and denote the point $(t_{i},t_{i}^{2}) \in \Gamma$ by $P_{i}$. 

We assume that the initial state is the configuration $\mathcal{P} = \{P_{1}, \cdots, P_{n}\}$ on the plane as described in Fig.~\ref{exa-dyna}. Notice that no four points of $\{P_{1}, \cdots, P_{n}\}$ are placed on the same circle, see~\cite{ManturovNikonov} for details.

\begin{figure}[h!]
 \centering
 \includegraphics[width = 6cm]{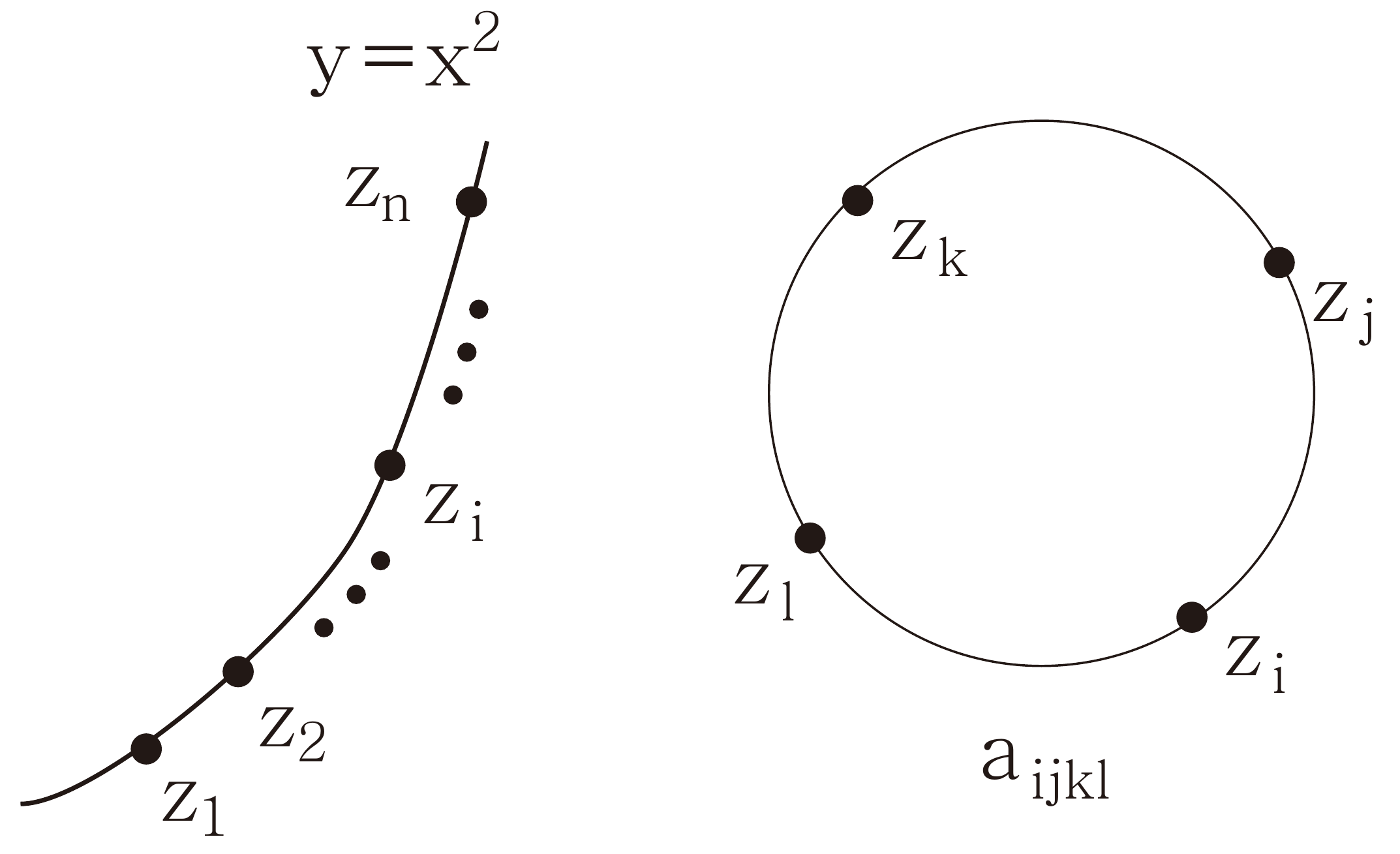}
 \caption{Initial state $\{P_{1}, P_{2}, \cdots, P_{n}\}$ such that $P_{i} = (t_{i},t_{i}^{2})$, where $\{t_{i}\}_{i=1}^{n}$ is a strictly increasing sequence}\label{exa-dyna}
\end{figure}

For any $i<j$ the pure braid $b_{ij}$ can be presented as the following dynamical system: the point $i$ moves along the graph $\Gamma$

\begin{enumerate}
\item the point $P_{i}(t)$ moves along the graphs $\Gamma$ and passes points 
$$ P_{i+1}(t), \cdots, P_{j}(t)$$from above, see the upper left of Fig.~\ref{moving_points__bij}.
\item the point $P_{j}(t)$ moves from above the point $P_{i}(t)$, see the upper right of Fig.~\ref{moving_points__bij}.
\item the point $P_{i}(t)$ moves to its initial position from above the points
$$ P_{j-1}(t), \cdots, P_{i+1}(t),$$
see the under of Fig.~\ref{moving_points__bij} 
\item the points $P_{j}(t)$ returns to the initial position.
\end{enumerate}

\begin{figure}[h!]
 \centering
 \includegraphics[width = 6cm]{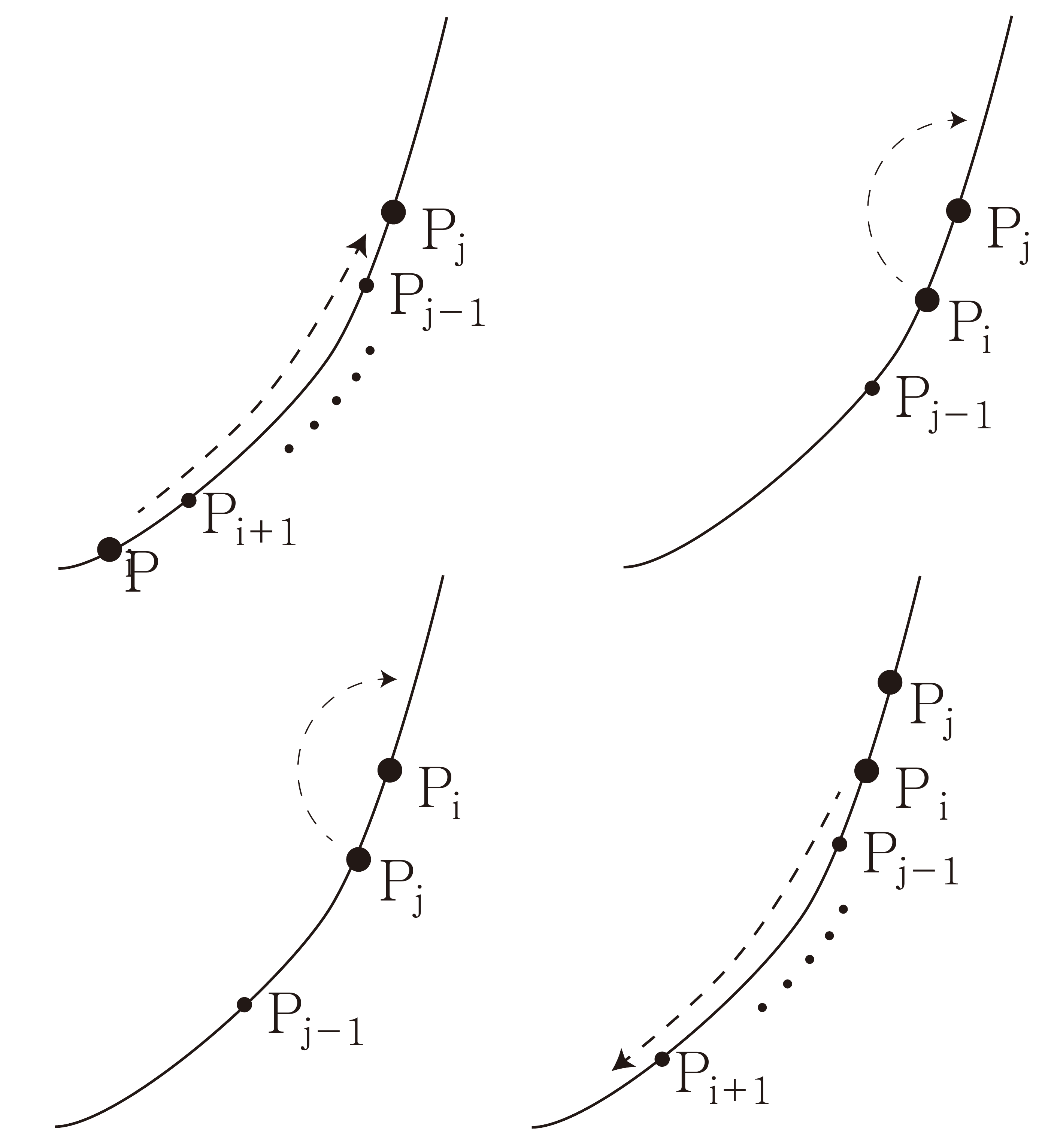}
 \caption{Model of moving points, corresponding to $b_{ij}$}\label{moving_points__bij}
\end{figure}

 Let $b_{ij} \in PB_{n}$, $1 \leq i<j \leq n$ be a generator. Consider the elements
\begin{eqnarray}
c_{ij}^{I} = \prod_{p=2}^{j-1}\prod_{q=1}^{p-1}a_{ijpq},\\
c_{ij}^{II} = \prod_{p=1}^{j-1}\prod_{q=1}^{n-j}a_{i(j-p)j(j+p)},\\
c_{ij}^{III} = \prod_{p=1}^{n-j+1}\prod_{q=0}^{n-p+1}a_{ij(n-p)(n-q)},
\end{eqnarray}
\begin{equation}\label{cij_homo_Gn4}
c_{ij}=c_{ij}^{II}c_{ij}^{I}c_{ij}^{III}.
\end{equation}

Notice that the formula~\ref{cij_homo_Gn4} is obtained by writing every moment, when four points are on the same circle, during the point $P_{i}$ passes over the point $P_{j}$ as described in Fig.~\ref{moving_points__cij}.
\begin{figure}[h!]
 \centering
 \includegraphics[width = 4cm]{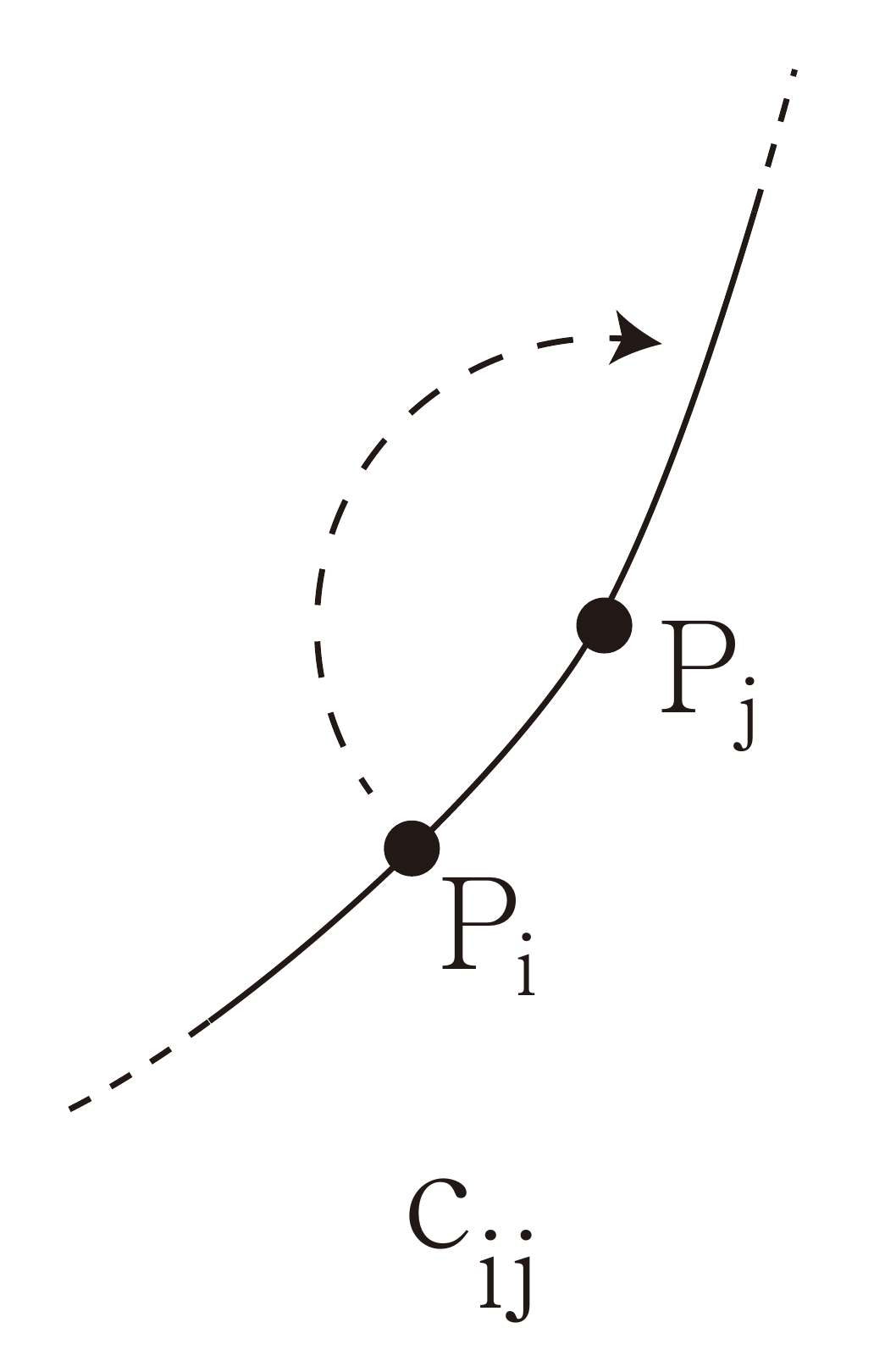}
 \caption{Model of moving points, corresponding to $c_{ij}$}\label{moving_points__cij}
\end{figure}

Now we define $\phi_{n} : PB_{n} \rightarrow G_{n}^{4}$ by
\begin{equation}\label{formula_homo_Gn4}
\phi_{n}(b_{ij}) = c_{i(i+1)}\cdots c_{i(j-1)}c_{ij}c_{ij}c_{i(j-1)}^{-1}\cdots c^{-1}_{i(i+1)},
\end{equation}
for $1 \leq i<j \leq n$.

\begin{prop}[\cite{ManturovNikonov}]
The mapping $\phi_{n}$ is well-defined.
\end{prop}

The above proposition can be proved by the basic principle, which is introduced in \cite{Manturov}.

\begin{dfn}\label{dfn_Gamma}
The group $\Gamma_{n}^{4}$ is the group given by group presentation generated by $\{ d_{(ijkl)}~|~ \{i,j,k,l\} \subset \bar{n}, |\{i,j,k\}| = 4\}$ subject to the following relations:

\begin{enumerate}
\item $d_{(ijkl)}^{2} = 1$ for $(i,j,k,l) \subset \bar{n}$, 
\item $d_{(ijkl)}d_{(stuv)} = d_{(stuv)}d_{(ijkl)}$, for $| \{i,j,k,l\} \cap \{s,t,u,v\} | < 3$,
\item $d_{(ijkl)}d_{(ijkm)}d_{(ijlm)}d_{(iklm)}d_{(jklm)} = 1$ for distinct $i,j,k,l,m$.
\item $d_{(ijkl)}=d_{(kjil)}=d_{(ilkj)}=d_{(klij)}=d_{(jkli)}=d_{(jilk)}=d_{(lkji)}=d_{(lijk)}$ for distinct $i,j,k,l,m$.
\end{enumerate}
\end{dfn}

The group $\Gamma_{n}^{4}$ is naturally related to  a triangulations of $2$-surfaces\footnote{The method presented here works for arbitrary $2$-surfaces as well but we restrict ourselves to the case of the plane} and the Pachner moves for the two dimensional case, called ``flip'', see Fig.~\ref{flip}. More precisely, a generator $d_{(ijkl)}$ of $\Gamma_{n}^{4}$ corresponds to the sequence of flips constituting the Pentagon relation, the most important relation for the group $\Gamma_{n}^{4}$.

\begin{figure}[h!]
 \centering
 \includegraphics[width = 8cm]{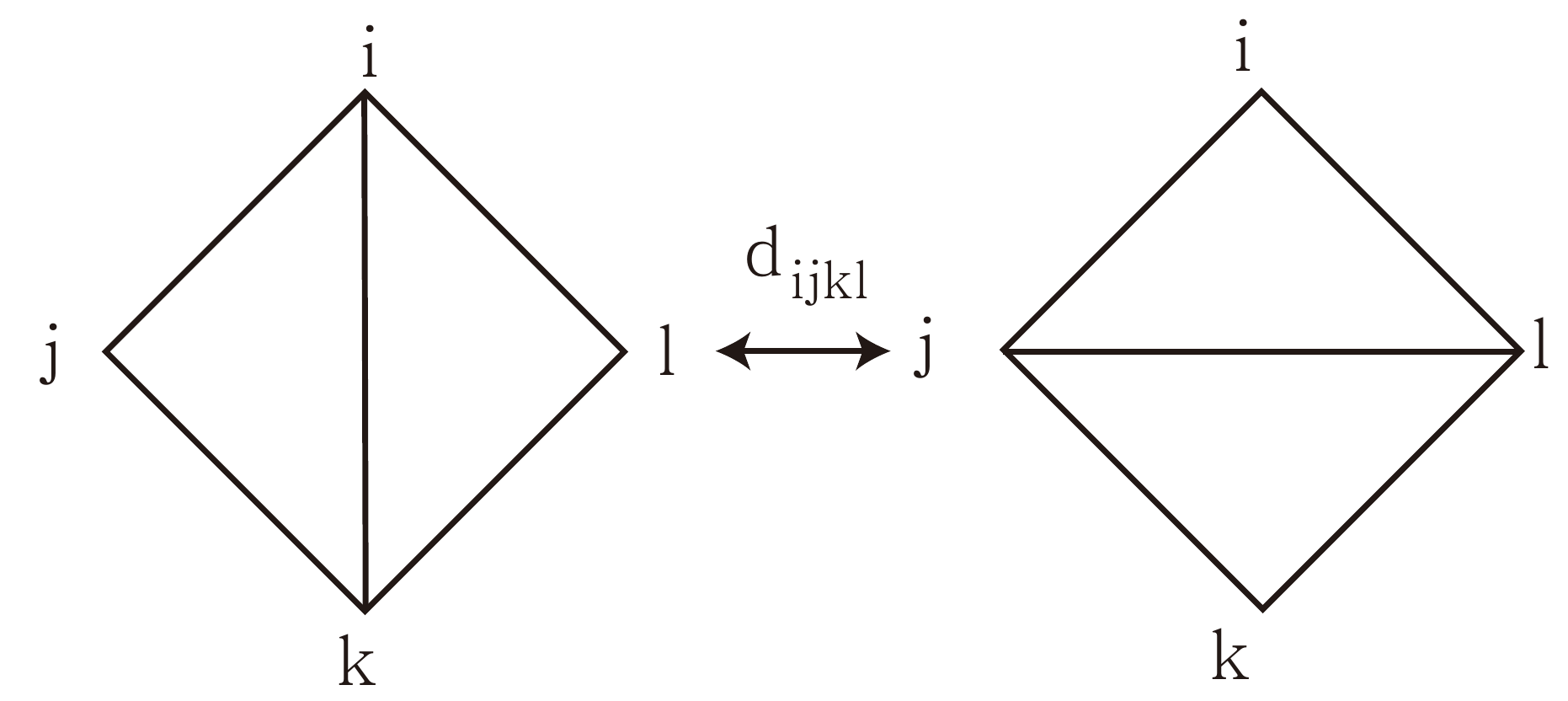}
 \caption{Flip on a rectangle $\Box ijkl$}\label{flip}
\end{figure}

Especially, the relation $d_{(ijkl)}d_{(ijkm)}d_{(ijlm)}d_{(iklm)}d_{(jklm)} = 1$ corresponds to the flips, applied on a pentagon as described in Fig.~\ref{flip_pentagon}.

\begin{figure}[h!]
 \centering
 \includegraphics[width = 10cm]{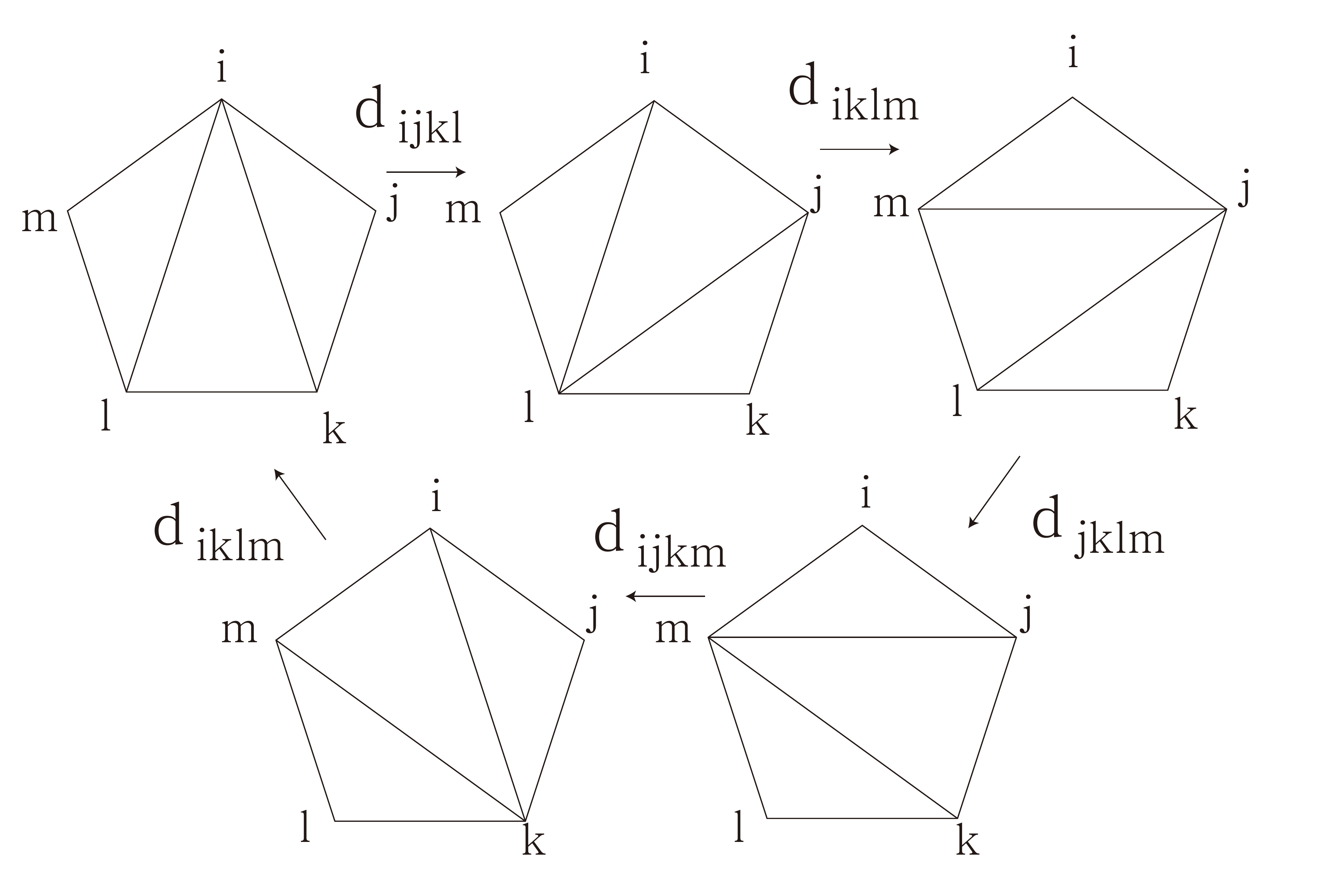}
 \caption{Flips on a pentagon $ijklm$}\label{flip_pentagon}
\end{figure}

\section{A group homomorphism from $PB_{n}$ to $\Gamma_{n}^{4}$}

In this section we construct the group homomorphism $f_{n}$ from $PB_{n}$ to $\Gamma_{n}^{4}$. The topological background for that is very easy: we consider codimension 1 ``walls'' which correspond to generators (flips) and codimension 2 relations (of the group $\Gamma$). Having this, we construct a map on the level of generators and prove its correctedness.

{\bf Geometric description of the mapping from $PB_{n}$ to $\Gamma_{n}^{4}$.}

Let us consider $b_{ij}$ as the dynamical system, described in Section 2.
The homomorphism $f_{n}$ from $PB_{n}$ to $\Gamma_{n}^{4}$ can be defined as follow; for the above dynamical system for each generator $b_{ij}$, let us enumerate $0<t_{1}<t_{2}< \cdots < t_{l} <1$ such that at the moment $t_{k}$ four points belong to the one circle. At the moment $t_{k}$, if $P_{s}, P_{t}, P_{u}, P_{v}$ are positioned on the one circle as the indicated order, then $d_{k} = d_{(stuv)}$. With the pure braid $b_{ij}$ we associate the product $f_{n}(b_{ij}) = d_{1}d_{2}\cdots d_{l}$.

{\bf Algebraic description of the mapping from $PB_{n}$ to $\Gamma_{n}^{4}$.}

On the other hand, the mapping $f_{n} : PB_{n} \rightarrow \Gamma_{n}^{4}$ can be formulated as follows:
Let us denote
\begin{center}
$d_{\{p,q, (r,s)_{s}\}}  = \left\{
\begin{array}{cc} 
     d_{(pqrs)} & \text{if}~p<q<s, \\
      d_{(prsq)} & \text{if}~p<s<q, \\ 
  d_{(rspq)} & \text{if}~s<p<q,\\
  d_{(qprs)} & \text{if}~q<p<s, \\
      d_{(qrsp)} & \text{if}~q<s<p, \\ 
  d_{(rsqp)} & \text{if}~s<q<p.
   \end{array}\right.$
   \end{center}
 \begin{rem}
Notice that the generator $d_{\{p,q, (r,s)_{s}\}}$ corresponds to four points $P_{p},P_{q},P_{r},P_{s}$ such that they are placed on a circle according to the order of $p,q,s$ and the point $P_{r}$ is placed close to $P_{s}$ for the orientation $P_{r}$ to $P_{s}$ to be the counterclockwise orientation, see Fig.~\ref{exa_pts_circle}. The subscription $s$ of $(r,s)_{s}$ means that the point $P_{s}$ does not move, but the point $P_{r}$ will move turning around the point $P_{s}$ after this moment. In other words, when we use the notation $d_{\{p,q, (r,s)_{s}\}}$, we are looking that the point $P_{r}$ is ``moving'' closely to the point $P_{s}$, turning around $P_{s}$. We would like to highlight that $d_{\{p,q, (r,s)_{s}\}} \neq d_{\{p,q, (s,r)_{s}\}}$.

\begin{figure}[h!]
 \centering
 \includegraphics[width = 8cm]{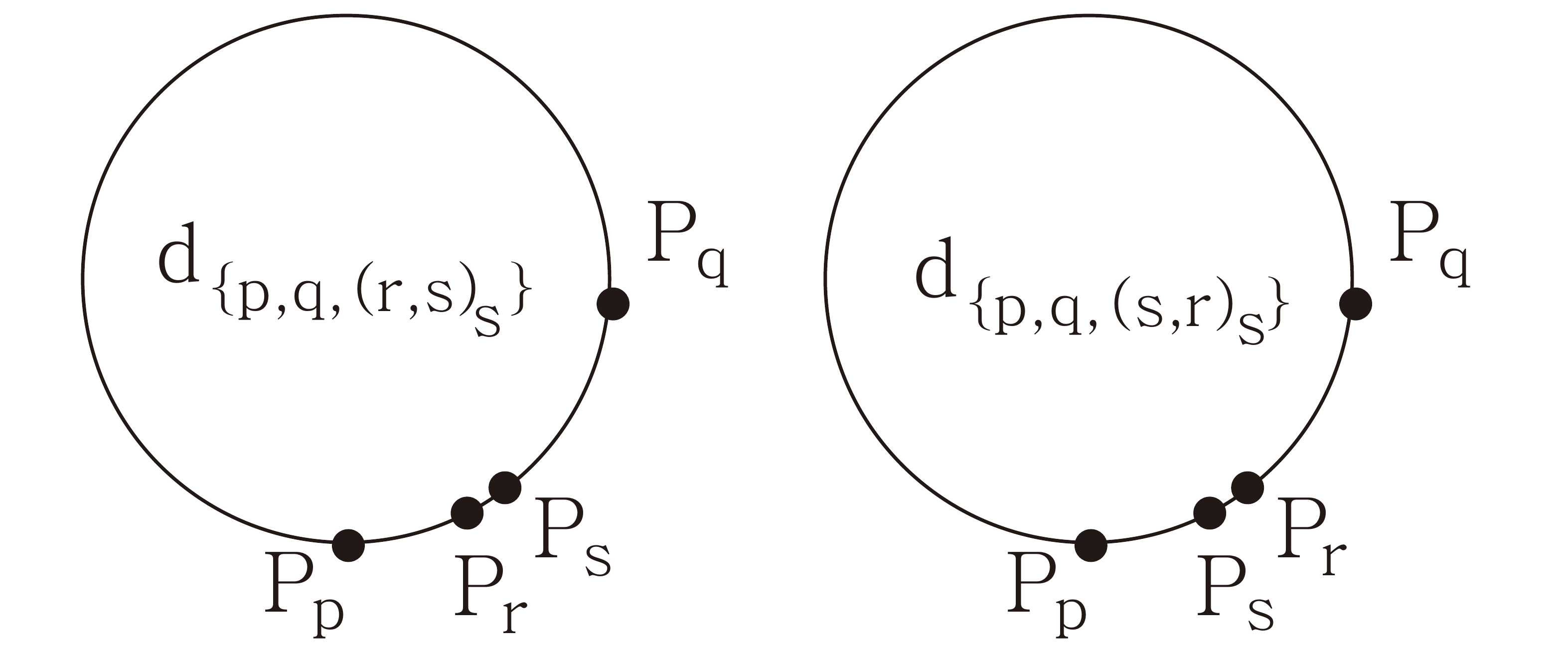}
 \caption{For $p<s<q$, $d_{\{p,q, (r,s)_{s}\}} = d_{(prsq)}$, but $d_{\{p,q, (s,r)_{s}\}} = d_{(psrq)}$.}\label{exa_pts_circle}
\end{figure}
\end{rem}

Let $b_{ij} \in PB_{n}$, $1 \leq i<j \leq n$, be a generator. Consider the elements
\begin{eqnarray*}
\gamma_{i,(i,j)}^{I} = \prod_{p=2}^{j-1}\prod_{q=1}^{p-1}d_{\{p,q,(i,j)_{j}\}},\\
\gamma_{i,(i,j)}^{II} = \prod_{p=1}^{j-1}\prod_{q=1}^{n-j}d_{\{(j-p),(j+p),(i,j)_{j}\}},\\
\gamma_{i,(i,j)}^{III} = \prod_{p=1}^{n-j+1}\prod_{q=0}^{n-p+1}d_{\{(n-p),(n-q),(i,j)_{j}\}},\\
\gamma_{i,(i,j)}=\gamma_{i,(i,j)}^{II}\gamma_{i,(i,j)}^{I}\gamma_{i,(i,j)}^{III}.
\end{eqnarray*}
Now we define $f_{n} : PB_{n} \rightarrow \Gamma_{n}^{4}$ by
$$f_{n}(b_{ij}) = \gamma_{i,(i,(i+1))}\cdots \gamma_{i,(i,(j-1))}\gamma_{i,(i,j)}\gamma_{i,(j,i)}\gamma_{i,((j-1),i)}^{-1}\cdots \gamma^{-1}_{i,((i+1),i)},$$
for $1 \leq i<j \leq n$.

\begin{thm}\label{thm_to_gamma}
The map $f_{n} : PB_{n} \rightarrow \Gamma_{n}^{4}$, which is defined above, is a well defined homomorphism. 
\end{thm}

\begin{proof}
When we consider isotopies between two pure braids, it suffices to take into account only singularities of codimension at most two. Singularities of codimension one give rise to generators, and relations come from singularities of codimension two. Now we list the cases of singularities of codimension two explicitly.

\begin{enumerate}
\item One point moving on the plane is tangent to the circle, which passes through three points, see Fig.~\ref{proof_rel1}. This corresponds to the relation $d_{(ijkl)}^{2} =1$.

\begin{figure}[h!]
 \centering
 \includegraphics[width = 8cm]{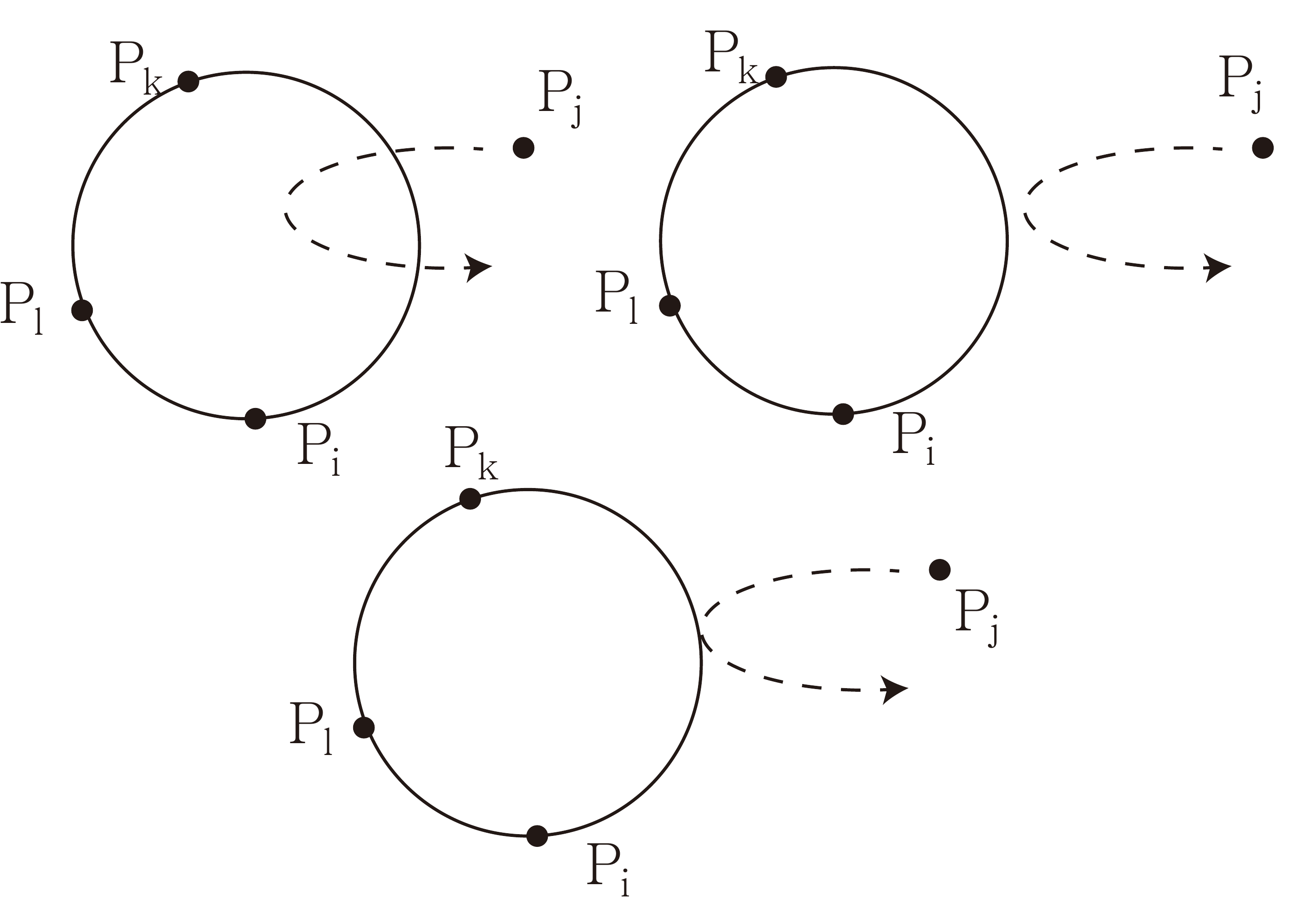}
 \caption{A point $P_{j}$ moves, being tangent to the circle, which passes through $P_{i},P_{k}$ and $P_{l}$}\label{proof_rel1}
\end{figure}

\item There are two sets $A$ and $B$ of four points, which are on the same circles such that $|A\cap B| \leq 2$, see Fig.~\ref{proof_rel2}. This corresponds to the relation $d_{(ijkl)}d_{(stuv)}=d_{(stuv)}d_{(ijkl)}$.

\begin{figure}[h!]
 \centering
 \includegraphics[width = 6cm]{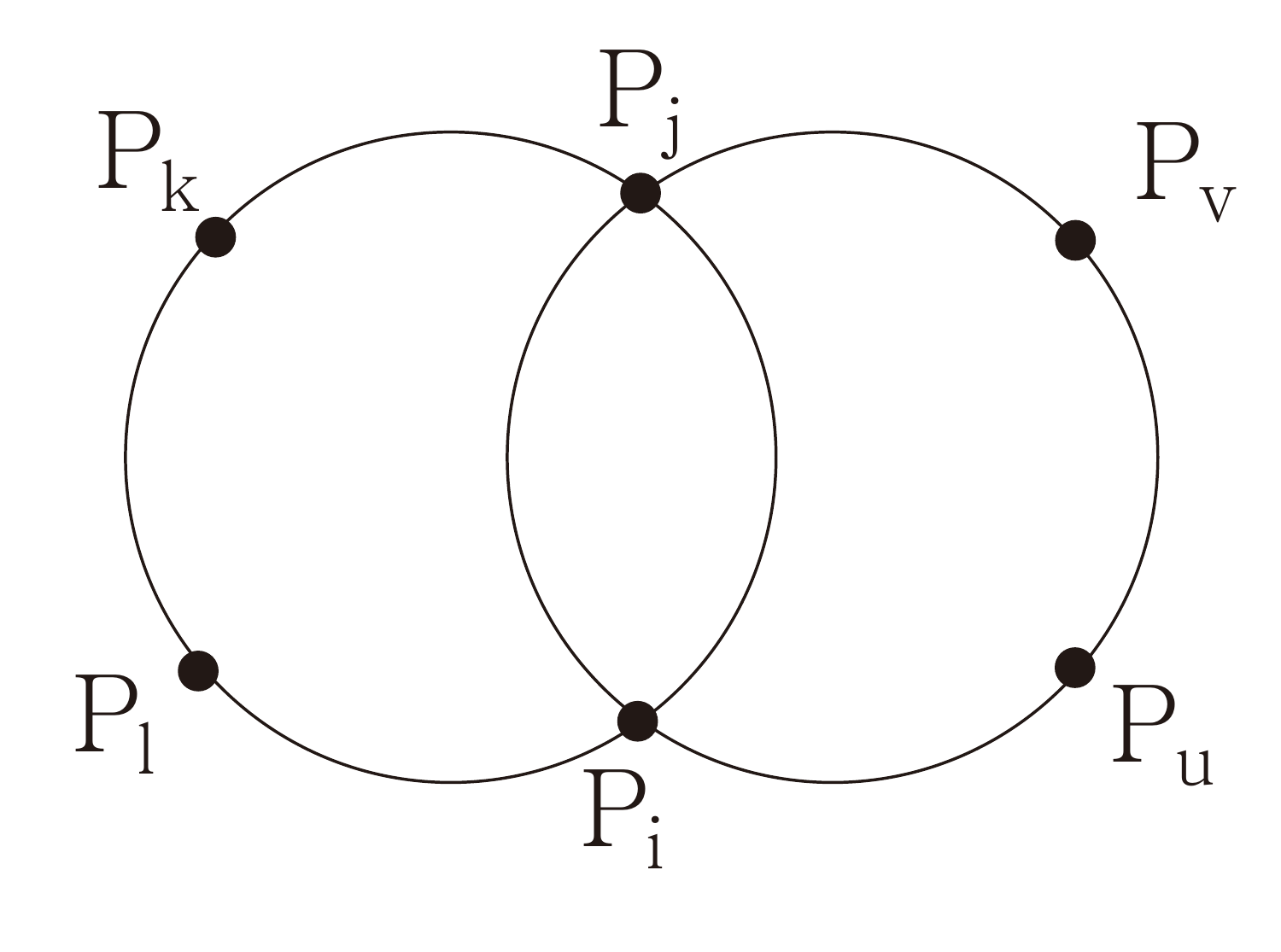}
 \caption{Two sets $A$ and $B$ of four points on the circles such that $|A\cap B| =2$}\label{proof_rel2}
\end{figure}

\item There are five points $\{P_{i},P_{j},P_{k},P_{l},P_{m}\}$ on the same circle. We obtain the sequence of five subsets of $\{P_{i},P_{j},P_{k},P_{l},P_{m}\}$ with four points on the same circle, which corresponds to the flips on the pentagon, see Fig.~\ref{proof_rel3}. This corresponds to the relation $d_{(ijkl)}d_{(ijkm)}d_{(ijlm)}d_{(iklm)}d_{(jklm)}=1$.

\begin{figure}[h!]
 \centering
 \includegraphics[width = 8cm]{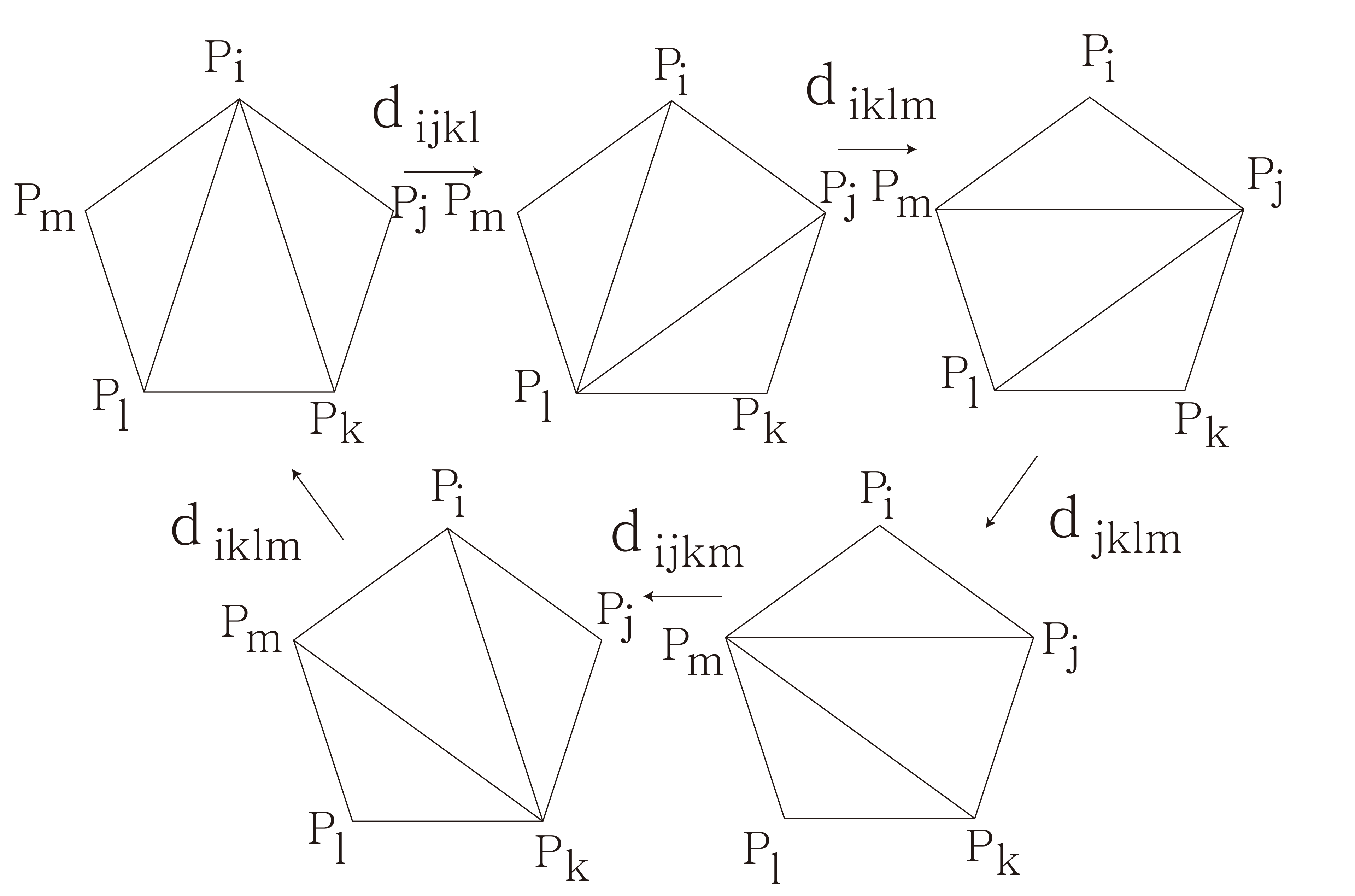}
 \caption{A sequence of five subsets of $\{P_{i},P_{j},P_{k},P_{l},P_{m}\}$, corresponding to flips on the pentagon}\label{proof_rel3}
\end{figure}
\end{enumerate}

\end{proof}

\section{A group homomorphism from $PB_{n}$ to $\Gamma_{n}^{4} \times \Gamma_{n}^{4}$ }
In this section we construct the group homomorphism $f^{2}_{n}$ from $PB_{n}$ to $\Gamma_{n}^{4} \times \Gamma_{n}^{4}$.
Roughly speaking, $f^{2}_{n}$ will be defined by reading generators of $\Gamma_{n}^{4} \times \Gamma_{n}^{4}$, which correspond to four points on the same circle, but we distinguish them with respect to the number of points inside the circle.

{\bf Geometric description of the mapping from $PB_{n}$ to $\Gamma_{n}^{4} \times \Gamma_{n}^{4}$.}

Let us consider the dynamical system for a generator $b_{ij}$, which is described in Section 2. Assume that an orientation on the plane is given. Let us enumerate $0<t_{1}<t_{2}< \cdots < t_{l} <1$ such that at the moment $t_{k}$ four points are positioned on the one circle (or on the line). Notice that by the assumption of the dynamical system for a generator $b_{ij}$, there are no four points on the circle at the initial. Let us assume that at the moment $t_{k}$, if $P_{s}, P_{t}, P_{u}, P_{v}$ are positioned on the one circle in the indicated order. If there are even number of points inside the circle, on which four points $P_{s}, P_{t}, P_{u}, P_{v}$ are placed, then $t_{k}$ corresponds to $\delta_{k} = (d_{(stuv)}, 1) \in \Gamma_{n}^{4} \times \Gamma_{n}^{4}$. Otherwise, $t_{k}$ corresponds to $\delta_{k} = ( 1,d_{(stuv)}) \in \Gamma_{n}^{4} \times \Gamma_{n}^{4}$. With the pure braid $b_{ij}$ we associate the product $f^{2}_{n}(b_{ij}) = \delta_{1}\delta_{2}\cdots \delta_{l}$.

\begin{figure}[h!]
 \centering
 \includegraphics[width = 8cm]{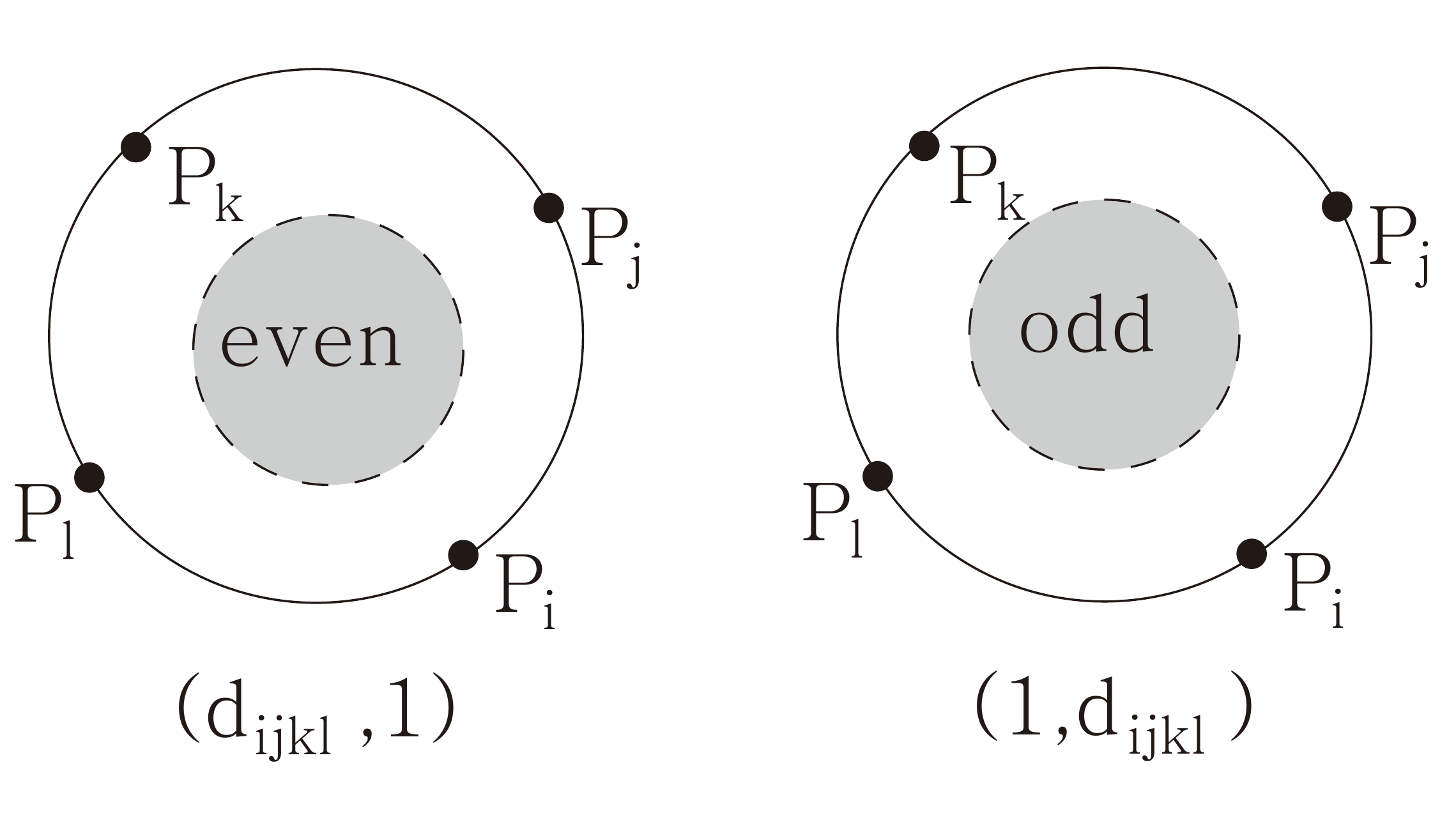}
 \caption{Geometric description for $f^{2}_{n}$}\label{geometric_description_psi}
\end{figure}

\begin{rem}
From the construction of base points $\{P_{1},\cdots, P_{n}\}$ it follows that the circle passing $P_{j},P_{p},P_{q}$ for $j<p<q$ contains points $\{P_{1}, \cdots, P_{j-1}\} \cup \{P_{p+1}, \cdots P_{q-1}\}$, see Fig.~\ref{circle}. That is, inside the circle passing $P_{j},P_{p},P_{q}$ for $j<p<q$ there are $j-1+q-p-1$ points from $\{P_{1},\cdots, P_{n}\}$.
\begin{figure}[h!]
 \centering
 \includegraphics[width = 8cm]{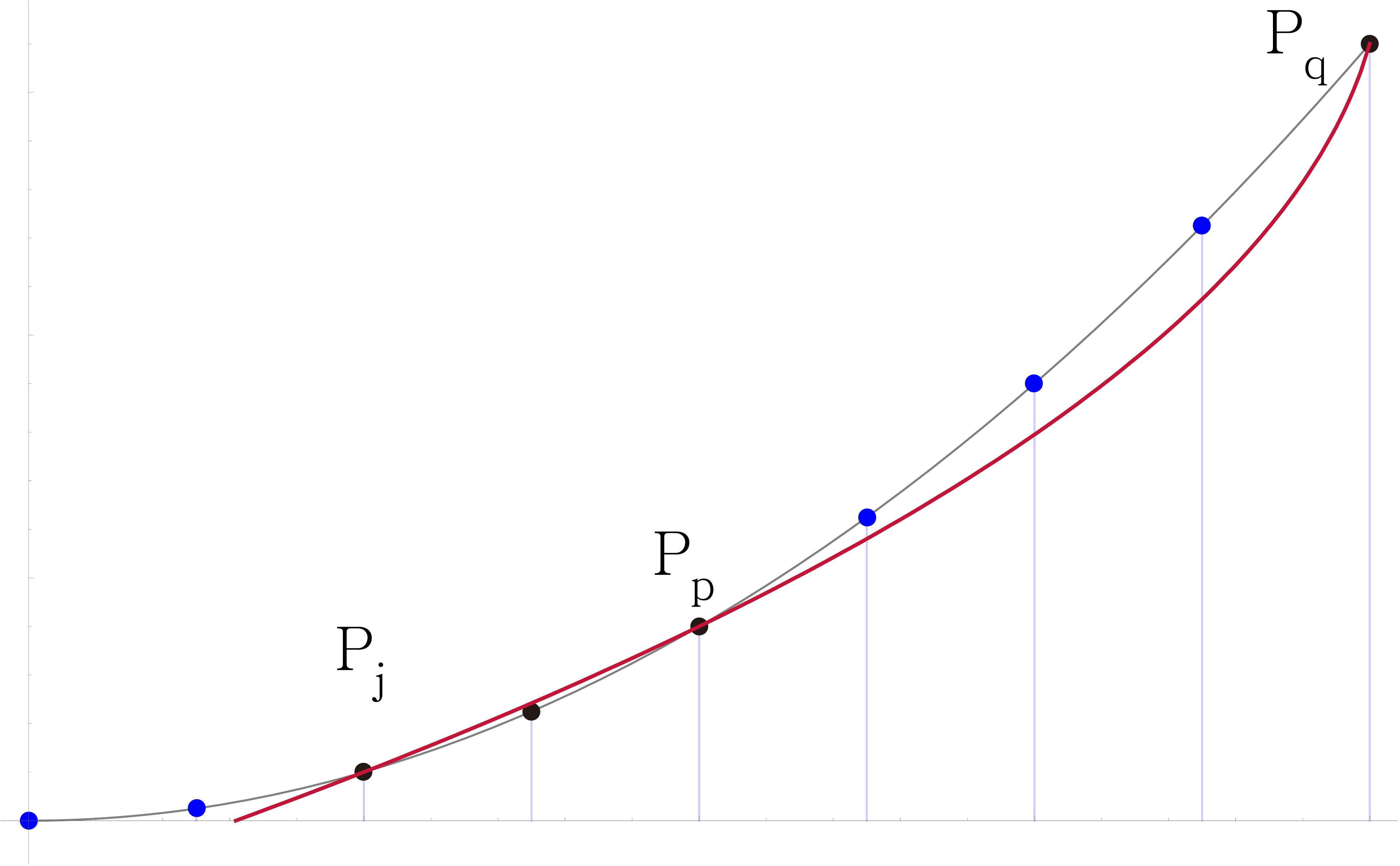}
 \caption{Red circle passes through points $P_{j},P_{p},P_{q}$ for $j<p<q$ and blue points are contained inside the red circle}\label{circle}
\end{figure}
\end{rem}

{\bf Algebraic description of the mapping from $PB_{n}$ to $\Gamma_{n}^{4} \times \Gamma_{n}^{4}$.}

On the other hands, the mapping $f^{2}_{n} : PB_{n} \rightarrow \Gamma_{n}^{4} \times \Gamma_{n}^{4}$ can be formulated as follow: Let us denote $mid\{p,q,r\}$ if $mid\{p,q,r\} \in \{p,q,r\}$ and $min\{p,q,r\}<mid\{p,q,r\}<max\{p,q,r\}$. Let us define $\delta_{\{p,q,(i,j)_{j}\}}$ as follows:

If $min\{p,q,j\}<i<mid\{p,q,r\}$ $i>max\{p,q,j\}$, then
\begin{center}
$\delta_{\{p,q,(i,j)_{j}\}}= \left\{
\begin{array}{cc} 
     (d_{\{p,q,(i,j)_{j}\}},1), & \text{if}~j+p+q\equiv 0~\text{mod}~2, \\
  (1,d_{\{p,q,(i,j)_{j}\}}),
 &   \text{if}~j+p+q \equiv 1~\text{mod}~2. \\
   \end{array}\right.$
\end{center}

If $i<min\{p,q,j\}$ or $mid\{p,q,j\}<i<max\{p,q,j\}$, then
\begin{center}
$\delta_{\{p,q,(i,j)_{j}\}}= \left\{
\begin{array}{cc} 
     (d_{\{p,q,(i,j)_{j}\}},1), & \text{if}~j+p+q\equiv 1~\text{mod}~2, \\
  (1,d_{\{p,q,(i,j)_{j}\}}),
 &   \text{if}~j+p+q \equiv 0~\text{mod}~2. \\
   \end{array}\right.$
\end{center}

 Let $b_{ij} \in PB_{n}$, $1 \leq i<j \leq n$ be a generator. Consider the elements

\begin{eqnarray}
D_{i,(i,j)}^{I} = \prod_{p=2}^{j-1}\prod_{q=1}^{p-1}\delta_{\{p,q,(i,j)_{j}\}},\\
D_{i,(i,j)}^{II} = \prod_{p=1}^{j-1}\prod_{q=1}^{n-j}\delta_{\{(j-p),(j+p),(i,j)_{j}\}},\\
D_{i,(i,j)}^{III} = \prod_{p=1}^{n-j+1}\prod_{q=0}^{n-p+1}\delta_{\{(n-p),(n-q),(i,j)_{j}\}},\\
D_{i,(i,j)}=D_{i,(i,j)}^{II}D_{i,(i,j)}^{I}D_{i,(i,j)}^{III}.
\end{eqnarray}

Now we define $f^{2}_{n} : PB_{n} \rightarrow \Gamma_{n}^{4} \times \Gamma_{n}^{4}$ by
$$f^{2}_{n}(b_{ij}) = D_{i,(i,(i+1))}\cdots D_{i,(i,(j-1))}D_{i,(i,j)}D_{i,(j,i)}D_{i,((j-1),i)}^{-1}\cdots D^{-1}_{i,((i+1),i)}$$
for $1 \leq i<j \leq n$.

\begin{thm}
The map $f^{2}_{n} : PB_{n} \rightarrow \Gamma_{n}^{4} \times \Gamma_{n}^{4}$, which is defined as above, is a well defined homomorphism. 
\end{thm}

\begin{proof}
This statement follows from Theorem~\ref{main_theorem_rgamma}.
\end{proof}

\section{A group homomorphism from $PB_{n}$ to $\Gamma_{n}^{4} \times \Gamma_{n}^{4} \times \cdots \times \Gamma_{n}^{4}$ }

The homomorphism $f^{2}_{n} : PB_{n} \rightarrow \Gamma_{n}^{4} \times \Gamma_{n}^{4}$ can be extended to a mapping from $PB_{n}$ to the product of $\Gamma_{n}^{4}$ of $r-$copies for $r>2$ as follow. We denote the product of $r$ copies of the group $\Gamma_n^4$ for $r>2$ by
$$\Gamma_{n,1}^{4} \times \Gamma_{n,2}^{4} \times \cdots \times \Gamma_{n,r}^{4}.$$
The idea is that we can not only distinguish between ``evenly many points inside the circle'' or ``oddly many points inside the circle'', but also just count this number of points.
For the above dynamical system let us enumerate $0<t_{1}<t_{2}< \cdots < t_{l} <1$ such that at the moment $t_{k}$ four points $P_{s}, P_{t}, P_{u}, P_{v}$ are positioned on one circle. If the number of points inside the circle is $\alpha$ mod $r$, on which four points $P_{s}, P_{t}, P_{u}, P_{v}$ are placed, then $t_{k}$ corresponds to $\delta_{k} = (1,\cdots ,1,d_{stuv}, 1, \cdots, 1) \in  \{1\} \times \cdots \{1\} \times \Gamma_{n,\alpha}^{4} \times \{1\} \times \cdots \{1\} \subset   \Gamma_{n,1}^{4} \times \Gamma_{n,2}^{4} \times \cdots \times \Gamma_{n,r}^{4}$. With the pure braid $b_{ij}$ we associate the product $f^{r}_{n}(b_{ij}) = \delta_{1}\delta_{2}\cdots \delta_{l}$.

Algebraically this construction can be presented as follows:

\begin{eqnarray*}
D_{i,(i,j)}^{I} = \prod_{p=2}^{j-1}\prod_{q=1}^{p-1}\delta^{r}_{\{p,q,(i,j)_{j}\}},\\
D_{i,(i,j)}^{II} = \prod_{p=1}^{j-1}\prod_{q=1}^{n-j}\delta^{r}_{\{(j-p),(j+p),(i,j)_{j}\}},\\
D_{i,(i,j)}^{III} = \prod_{p=1}^{n-j+1}\prod_{q=0}^{n-p+1}\delta^{r}_{\{(n-p),(n-q),(i,j)_{j}\}},\\
\end{eqnarray*}
where 
\begin{eqnarray*}
\delta^{r}_{\{p,q,(i,j)_{j}\}}&= &
    (1, \cdots, d_{\{p,q,(i,j)_{j}\}},1, \cdots, 1) \in  \{1\} \times \cdots \{1\} \times \Gamma_{n,\alpha}^{4} \times \{1\} \times \cdots \{1\}  \\  && \subset \Gamma_{n,1}^{4} \times \Gamma_{n,2}^{4} \times \cdots \times \Gamma_{n,r}^{4},
    \end{eqnarray*}
if     
\begin{center}
$min\{j,p,q\}-mid\{j,p,q\}+max\{j,p,q\} -2 \equiv \left\{
\begin{array}{cc} 
    \alpha~mod~r, & \text{if}~min\{p,q,j\}<i<mid\{p,q,r\}~\text{or}~i>max\{p,q,j\}, \\
   \alpha-1~mod~r,
 &   \text{if}~i<min\{p,q,j\} ~\text{or}~ mid\{p,q,j\}<i<max\{p,q,j\}. \\
   \end{array}\right.$
\end{center}
     
$$D_{i,(i,j)}=D_{i,(i,j)}^{II}D_{i,(i,j)}^{I}D_{i,(i,j)}^{III}.$$
Now we define $f^{r}_{n} : PB_{n} \rightarrow \Gamma_{n,1}^{4} \times \Gamma_{n,2}^{4} \times \cdots \times \Gamma_{n,r}^{4}$ by
$$\psi_{n}(b_{ij}) = D_{i,(i,(i+1)}\cdots D_{i,(i,(j-1))}D_{i,(i,j)}D_{i,(i,j)}D_{(i,(i,(j-1))}^{-1}\cdots D^{-1}_{i,(i,(i+1)},$$
for $1 \leq i<j \leq n$.

\begin{thm}\label{main_theorem_rgamma}
The map $f^{r}_{n} : PB_{n} \rightarrow \Gamma_{n,1}^{4} \times \Gamma_{n,2}^{4} \times \cdots \times \Gamma_{n,r}^{4}$, which is defined as above, is a well defined homomorphism. 
\end{thm}

\begin{proof}
This statement can be proved similarly to toe proof of Theorem~\ref{thm_to_gamma}. Let us list cases of singularities of codimension two explicitly. Notice that the image of $f^{r}_{n}$ from four points on the circle depends on the number of points inside the circle. 

\begin{enumerate}
\item One point moving on the plane is tangent to the circle, which passes through three points, see the center in Fig.~\ref{proof_rel1}. Notice that the number of points inside the circle does not change when the point $P_{j}$ moves. It is easy to see that the image, when one point passes through the circle twice (upper left in Fig.~\ref{proof_rel1}), is
\begin{eqnarray*}
 (1, \cdots,1, d_{(ijkl)}d_{(ijkl)},1, \cdots, 1) \in  \Gamma_{n,1}^{4} \times \cdots \times \Gamma_{n,\alpha}^{4} \times \cdots \times \Gamma_{n,r}^{4},
    \end{eqnarray*}
where the $\alpha$ is the number of points inside the circle, passing through $P_{i},P_{k},P_{l}$.
If the point does not pass through the circle (upper right in Fig.~\ref{proof_rel1}), then the image is
\begin{eqnarray*}
 (1, \cdots, 1, \cdots, 1) \in  \Gamma_{n,1}^{4} \times \cdots \times \Gamma_{n,\alpha}^{4} \times \cdots \times \Gamma_{n,r}^{4}.
    \end{eqnarray*}
 The equality of those two images is obtained by the relation $d_{(ijkl)}^{2} =1$.

\item There are two sets $A=\{P_{i},P_{j},P_{k},P_{l}\}$ and $B=\{P_{s},P_{t},P_{u},P_{v}\}$ of four points, which are on the same circles such that $|A\cap B| \leq 2$, see Fig.~\ref{proof_rel2}. If the number of points inside circles, which passes through points $\{P_{i},P_{j},P_{k},P_{l}\}$ and $\{P_{s},P_{t},P_{u},P_{v}\}$ respectively, are the same mod $r$, then the image from them is 
\begin{eqnarray*}
 (1, \cdots, d_{(ijkl)}d_{(stuv)},1, \cdots, 1) \in  \Gamma_{n,1}^{4} \times \cdots \times \Gamma_{n,\alpha}^{4} \times \cdots \times \Gamma_{n,r}^{4},
    \end{eqnarray*}
or 
\begin{eqnarray*}
 (1, \cdots, d_{(stuv)}d_{(ijkl)},1, \cdots, 1) \in  \Gamma_{n,1}^{4} \times \cdots \times \Gamma_{n,\alpha}^{4} \times \cdots \times \Gamma_{n,r}^{4},
    \end{eqnarray*}
The equality of them follows from the relation $d_{(ijkl)}d_{(stuv)} = d_{(stuv)}d_{(ijkl)}$.

If the number of points inside circles, where the points $\{P_{i},P_{j},P_{k},P_{l}\}$ and $\{P_{s},P_{t},P_{u},P_{v}\}$ are positioned respectively, are different mod $r$, then the image from them is 
\begin{eqnarray*}
 (1, \cdots, d_{(ijkl)}, \cdots, 1 , \cdots 1)  (1, \cdots,1, \cdots, d_{(stuv)}, \cdots 1) \\ \in  \Gamma_{n,1}^{4} \times \cdots \times \Gamma_{n,\alpha}^{4} \times \cdots \times \Gamma_{n,\beta}^{4} \times \cdots \times \Gamma_{n,r}^{4},
    \end{eqnarray*}    
 or
 \begin{eqnarray*}
 (1, \cdots,1, \cdots, d_{(stuv)}, \cdots 1) (1, \cdots, d_{(ijkl)}, \cdots, 1 , \cdots 1) \\ \in  \Gamma_{n,1}^{4} \times \cdots \times \Gamma_{n,\alpha}^{4} \times \cdots \times \Gamma_{n,\beta}^{4} \times \cdots \times \Gamma_{n,r}^{4},
    \end{eqnarray*}    
where $\alpha$ and $\beta$ depend on the numbers of points inside the circles, which pass through $\{P_{i},P_{j},P_{k},P_{l}\}$ and $\{P_{s},P_{t},P_{u},P_{v}\}$, respectively. It is easy to obtain the equality of them.

\item There are five points $\{P_{i},P_{j},P_{k},P_{l},P_{m}\}$ on the same circle. We obtain the sequence of five subsets of $\{P_{i},P_{j},P_{k},P_{l},P_{m}\}$ with four points on the same circle, which corresponds to the flips on the pentagon, see Fig.~\ref{proof_rel3}. Notice that the number of points inside circles does not change. In other words, the image of the sequence has the form of
\begin{eqnarray*}
\indent\indent\indent(1, \cdots, d_{(ijkl)}d_{(ijkm)}d_{(ijlm)}d_{(iklm)}d_{(jklm)},1, \cdots, 1) \in  \Gamma_{n,1}^{4} \times \cdots \times \Gamma_{n,\alpha}^{4} \times \cdots \times \Gamma_{n,r}^{4},
    \end{eqnarray*}
or
\begin{eqnarray*}
\indent\indent\indent (1, \cdots, d_{(jklm)}d_{(iklm)}d_{(ijlm)}d_{(ijkm)}d_{(ijkl)},1, \cdots, 1) \in  \Gamma_{n,1}^{4} \times \cdots \times \Gamma_{n,\alpha}^{4} \times \cdots \times \Gamma_{n,r}^{4}.
    \end{eqnarray*}
The equality of them is obtained from the relation $$d_{(ijkl)}d_{(ijkm)}d_{(ijlm)}d_{(iklm)}d_{(jklm)}=1 = d_{(jklm)}d_{(iklm)}d_{(ijlm)}d_{(ijkm)}d_{(ijkl)}$$ of $\Gamma_{n}^{4}$.

\end{enumerate}

\end{proof}

\section{Braids in $\mathbb{R}^{3}$ and groups $\Gamma_{n}^{4}$}

In~\cite{Manturov_Gnk_config}, the second named author introduced the notion of braids for $\mathbb{R}^{3}$ and $\mathbb{R}P^{3}$. Roughly speaking, {\it a braid for $\mathbb{R}^{3}$(or $\mathbb{R}P^{3}$)} is a path in a configuration space $C'_{n}(\mathbb{R}^{3})$ (or $C'_{n}(\mathbb{R}P^{3})$) with some restrictions. If the initial and end points of the path in $C'_{n}(\mathbb{R}^{3})$ coincide, then the path is called {\it a pure braid for $\mathbb{R}^{3}$($\mathbb{R}P^{3}$)}. In the present section we will construct a group homomorphism from pure braids on $n$ strands in $\mathbb{R}^{3}$ to $\gamma_{n}^{4}$. 


Let us recall the definition of the pure braids for $\mathbb{R}^{k-1}$ for $k>3$.
Let $C'_{n}(\mathbb{R}^{k-1})$ be the subset of $C_{n}(\mathbb{R}^{k-1})$ of all points $x = (x_{1}, \cdots, x_{n})$ such that no $k-1$ points among $\{x_{1}, \cdots, x_{n}\}$ are on the same $(k-3)-$plane.
We say that a point $x \in C'_{n}(\mathbb{R}^{k-1})$ is {\it singular}, if the set of points $x= (x_{1},\cdots, x_{n})$ representing it contains some $k$ points which are not belong to the same $(k-2)-$plane.

We call elements in $\pi_{n}(C'_{n}(\mathbb{R}^{k-1}))$ {\it pure braids on $n$ strands in $\mathbb{R}^{k-1}$}.

The group homomorphism from $\pi_{n}(C'_{n}(\mathbb{R}^{k-1}))$ to $G_{n}^{k}$ is constructed as follows:
Fix two non-singular points $x,x' \in C_{n}'(\mathbb{R}^{k-1})$. 
Let us consider the set of smooth paths $\gamma_{x,x'} : [0,1] \rightarrow C_{n}'(\mathbb{R}^{k-1})$. We call $t \in [0,1]$ {\it a singular moment of $\gamma_{x,x'}$}, if $\gamma_{x,x'}(t)$ is a singular point in $C_{n}'(\mathbb{R}^{k-1})$. We call a smooth path is {\it stable and good} if the following conditions hold:
\begin{enumerate}
\item The set of singular moments $t$ is finite;
\item For each singular moment $t=t_{l}$, there is only one subset of $k$ points belonging to a $(k-2)-$plane among $n$ points $x_{1},\cdots x_{n} \in \mathbb{R}^{k-1}$ such that $\gamma_{x,x'}(t_{l}) = (x_{1},\cdots, x_{n})$.
\item A smooth path is {\it stable}, if the number of singular moments does not change under a small perturbation.
\end{enumerate}

We say that two paths $\gamma,\gamma'$ with the same endpoints $x,x'$ are {\it isotopic}, if there exists a continuous family $\{\gamma_{x,x'}^{s}: [0,1] \rightarrow C'_{n}(\mathbb{R}^{3})\}_{s\in [0,1]}$ of smooth paths with endpoints fixed, such that $\gamma_{x,x'}^{0} = \gamma$ and $\gamma_{x,x'}^{1} = \gamma'$. A smooth path with end points $x = (x_{1},\cdots, x_{n})$ and $x' = (x'_{1},\cdots, x'_{n})$ is called {\it a braid (or a pure braid) on $n$ strands in $\mathbb{R}^{k-1} $}, if 
$ \{ x_{1},\cdots, x_{n} \} = \{ x'_{1},\cdots, x'_{n}\}$ (or $(x_{1},\cdots, x_{n}) = (x'_{1},\cdots, x'_{n})$).

For paths $\gamma_{x,x'}$ and $\gamma_{x',x''}$ the {\it concatenation} operation is well-defined, that is, a smooth path $\gamma_{x,x''}''$ such that $\gamma_{x,x''}''(t) = \gamma_{x,x'}(2t)$ for $t \in [0,\frac{1}{2}]$ and $\gamma_{x,x''}''(t) = \gamma_{x,x'}'(2t-1)$ for $t \in [\frac{1}{2},1]$ and smooth it in the neighborhood of $t= \frac{1}{2}$ is uniquely obtained from $\gamma_{x,x'}$ and $\gamma_{x',x''}$ up to isotopy. It is easy to see that the set of equivalence class of pure braids on $n$ strands in $\mathbb{R}^{k-1}$ up to isotopy admits a group structure, moreover it is isomorphic to $\pi_{1}(C'_{n}(\mathbb{R}^{k-1})$.

\begin{rem}
For a given pure braid $\gamma$ on $n$ strands in $\mathbb{R}^{k-1}$ by a small perturbation we can obtain a good and stable pure braid $\gamma'$ on $n$ strands such that $\gamma$ and $\gamma'$ are isotopic. From now on we just consider good and stable pure braids.
\end{rem}

Let $\gamma$ be a good and stable pure braid on $n$ strands. Let us enumerate all singular moments $0<t_{1}<\cdots< t_{l}<1$ of $\gamma$. For each $t_{s}$ by definition of good pure braids on $n$ strands there are exactly $k$ points on $(k-2)-$plane. Let $m_{s}$ be the set of $k$ indices for points $(k-2)-$plane at the moment $t_{s}$. We associate $a_{m_{s}}$ with the moment $t_{s}$. We define a map $f :  \pi_{1}(C_{n}'(\mathbb{R}^{k-1})) \rightarrow G_{n}^{k}$ by $f(\gamma) = a_{m_{1}}\cdots a_{m_{l}}$.

\begin{prop}\cite{Manturov_Gnk_config}
The map $f : \pi_{1}(C_{n}'(\mathbb{R}^{k-1})) \rightarrow G_{n}^{k}$ described above is a group homomorphism.
\end{prop}

We shall consider (good and stable) pure braids on $n$ strands in $\mathbb{R}^{3}$ and construct group homomorphism from pure braids on $n$ strands in $\mathbb{R}^{3}$ to the group $\Gamma_{n}^{4}$. Each element of $\Gamma_{n}^{4}$ corresponds to the moment when four points on $(4-2)$-dimensional plane in $\mathbb{R}^{4-1}$, but in the case of $\Gamma_{n}^{4}$ ``the order'' of four points on $(4-2)$-dimensional plane is very important. This order was ignored when the group homomorphism from pure braids in $\mathbb{R}^{3}$ to $G_{n}^{4}$ is constructed. Now we formulate more precisely how the group homomorphism from pure braids on $n$ strands in $\mathbb{R}^{3}$ to $G_{n}^{4}$ is constructed.

Let $\gamma$ be a good and stable pure braid on $n$ strands in $\mathbb{R}^{3}$ with base point $x = (x_{1},\cdots, x_{n})$. We call $t \in [0,1]$ {\it a special singular moment of $\gamma$} if the followings hold:

\begin{enumerate}
\item At the moment $t$ four points $x_{p},x_{q},x_{r},x_{s}$ are on the same plane $\Pi_{t}$.
\item The four points $x_{p},x_{q},x_{r},x_{s}$ make a convex quadrilateral on $\Pi_{t}$.
\item All of $\{x_{1},\cdots, x_{n}\} \backslash \{ x_{p},x_{q},x_{r},x_{s} \}$ placed in the only one of connected components of $\mathbb{R}^{3} \backslash \Pi_{t}$.
\end{enumerate}

A normal vector $v_{t}$ of $\Pi_{t}$ pointing to the connected component, in which the points $\{x_{1},\cdots, x_{n}\} \backslash \{ x_{p},x_{q},x_{r},x_{s} \}$ are placed, is called {\it the pointing vector} at the special singular moment $t$. 

\begin{figure}[h!]
 \centering
 \includegraphics[width = 5cm]{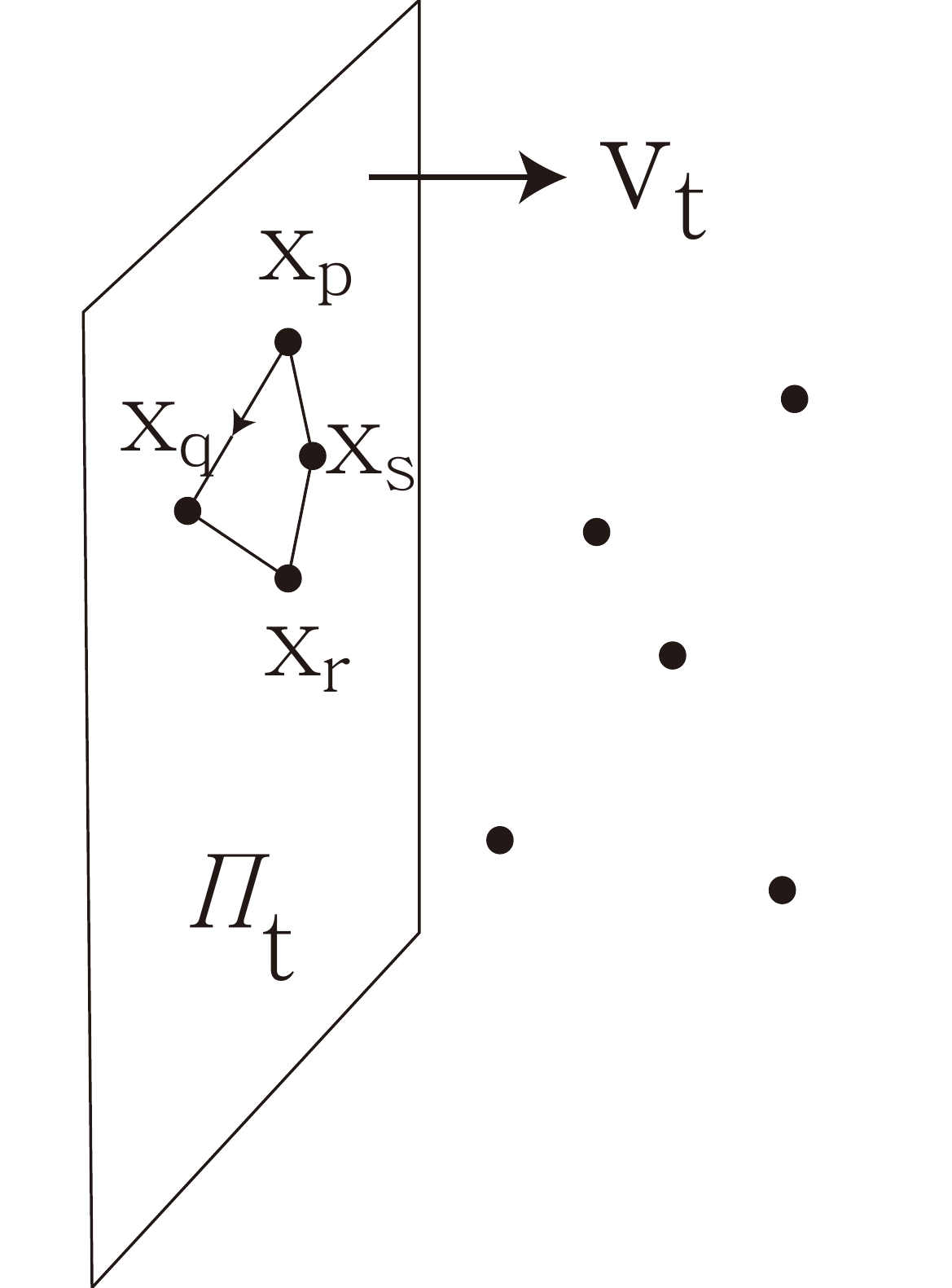}
 \caption{A special singular moment, corresponding to $d_{(pqrs)}$}\label{singular_moment}
\end{figure}

\begin{rem}

\end{rem}\begin{enumerate}
\item The plane $\Pi_{t}$ admits a unique orientation with respect to $v_{t}$. 
\item Naturally, the quadrilateral admits the orientation with respect to $v_{t}$, see Fig.~\ref{singular_moment}.
\end{enumerate}

Let us enumerate all special singular moments $0<t_{1}<\cdots< t_{l}<1$ of $\gamma$. For each $t_{s}$ by definition of good pure braids on $n$ strands there are exactly four points $\{ x_{p},x_{q},x_{r},x_{s}\}$ on plane $\Pi_{t_{s}}$. As indicated in the previous remark the convex quadrilateral on plane $\Pi_{t_{s}}$ with four vertices $\{ x_{p},x_{q},x_{r},x_{s}\}$ admits the orientation with respect to $v_{t_{s}}$. If four points $x_{p},x_{q},x_{r},x_{s}$ are positioned as indicated order in accordance the orientation with respect to $v_{t_{s}}$, then we associate the moment $t_{s}$ to $d_{t_{s}}=d_{(pqrs)}$. Let us define a map $g: \pi_{1}(C_{n}'(\mathbb{R}^{3})) \rightarrow \Gamma_{n}^{4}$ by
$g(\gamma) = d_{t_{1}}\cdots d_{t_{l}}$.

\begin{thm}
The map $g: \pi_{1}(C_{n}'(\mathbb{R}^{3})) \rightarrow \Gamma_{n}^{4}$ is well-defined.
\end{thm}

\begin{proof}
We consider moments of isotopy between two paths, when the path at some moment in the isotopy between two paths is not good or not stable. Let us list such cases explicitly.

\begin{enumerate}
\item There are four points on a , which disappears after a small perturbation, see Fig.~\ref{rel1_gamma}. This corresponds to the relation $d_{(pqrs)}^{2} = 1$. 

\begin{figure}[h!]
 \centering
 \includegraphics[width = 6cm]{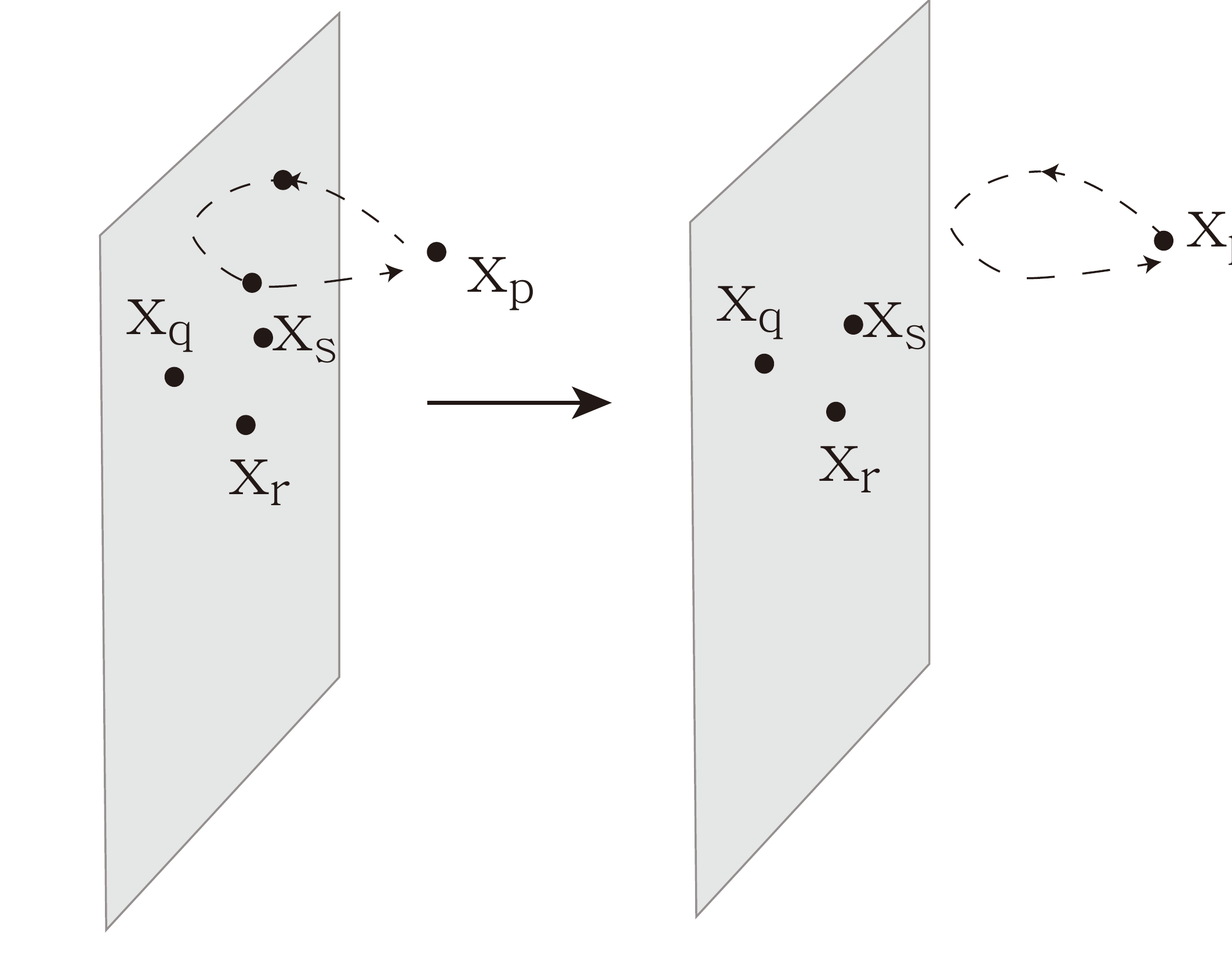}
 \caption{Case 1: Four points on a plane, which disappears after a small perturbation}\label{rel1_gamma}
\end{figure}

\item At a moment there are two sets of four points $m$ and $m'$ with $|m \cap m'|<3$, which are placed on planes at the same moment, see Fig.~\ref{rel2_gamma}. This corresponds to the relation $d_{(ijkl)}d_{(stuv)} =d_{(stuv)} d_{(ijkl)}$.

\begin{figure}[h!]
 \centering
 \includegraphics[width = 5cm]{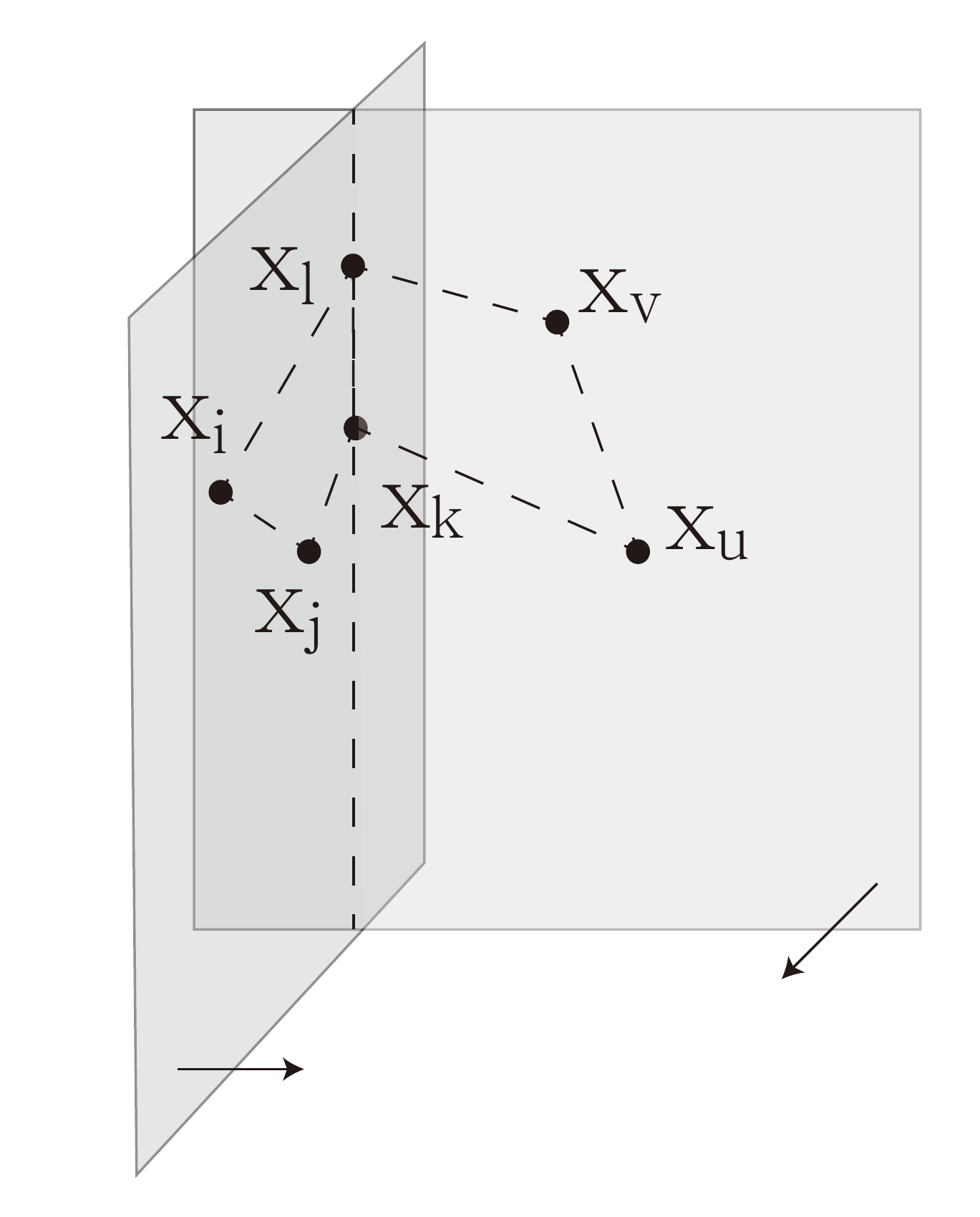}
 \caption{Case 2: Two sets of four points $\{x_{i},x_{j},x_{k},x_{l}\}$ and $\{x_{l},x_{k},x_{u},x_{v}\}$ on planes at the same moment}\label{rel2_gamma}
\end{figure}

\item At a moment five points on a plane. This is similar to the case of ``five points on the circle'' in the proof of Theorem~\ref{thm_to_gamma}. This corresponds to the relation $d_{(ijkl)}d_{(ijkm)}d_{(ijlm)}d_{(iklm)}d_{(jklm)} = 1$ of $\Gamma_{n}^{4}$.

\end{enumerate}

\end{proof}

Let $\{P_{1}(t), \cdots, P_{n}(t)\}_{t \in [0,1]}$ be $n$ moving points in $\mathbb{R}^{3}$, corresponding to the path in $\pi_{1}(C_{n}'(\mathbb{R}^{3}))$. We may assume that the points $\{P_{1}(t), \cdots, P_{n}(t)\}_{t \in [0,1]}$ move inside a the sphere with sufficiently large diameter. Let us fix four points $\{A,B,C,D\}$ on the sphere. A triangulation of 3-ball with vertices $\{P_{1}(t), \cdots, P_{n}(t)\} \cup \{A,B,C,D\}$ can be obtained for each $t \in [0,1]$, see Fig.~\ref{triangulation_R3}. 

\begin{figure}[h!]
 \centering
 \includegraphics[width = 6cm]{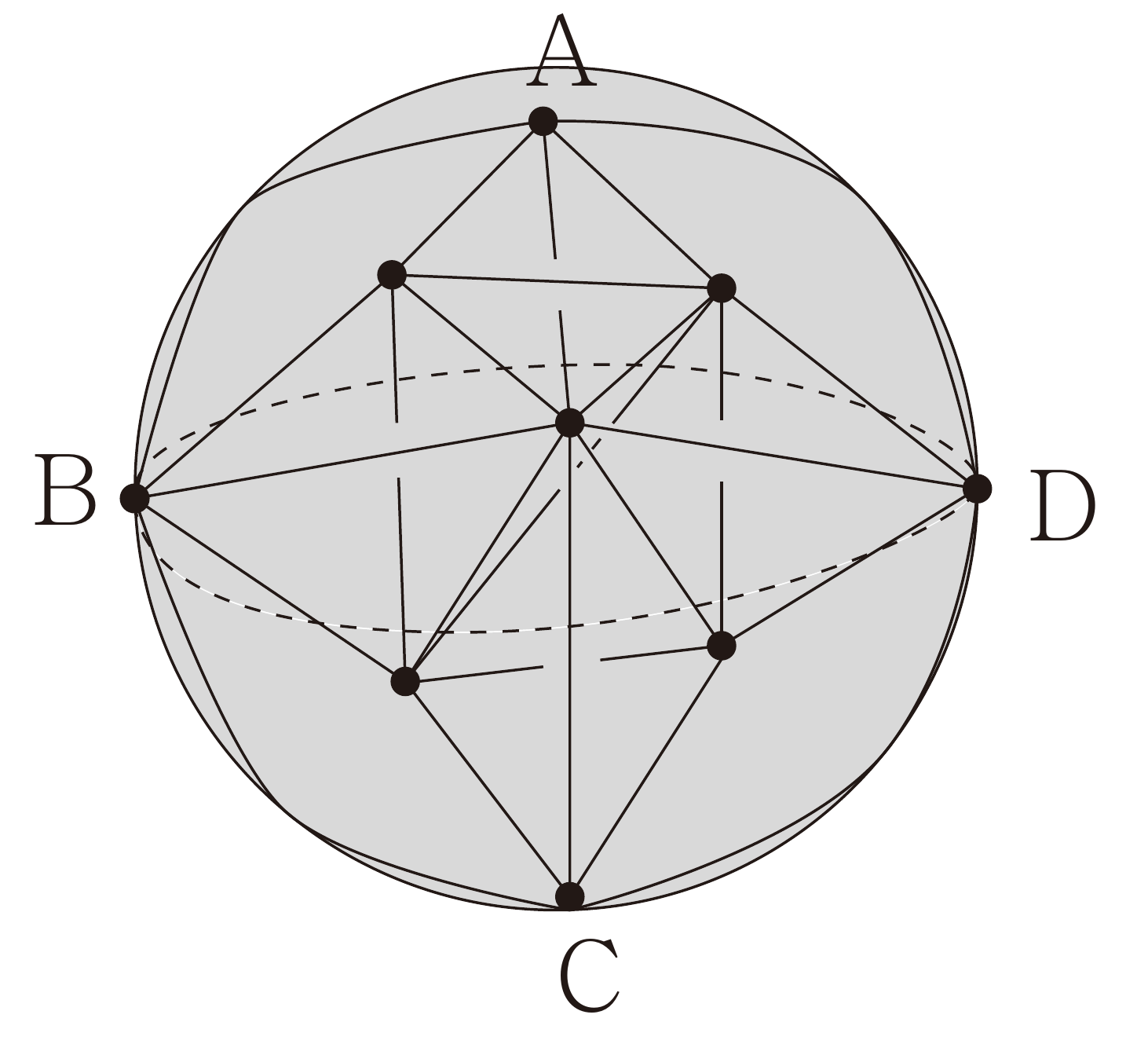}
 \caption{Triangulation of 3-disk with $\{P_{1}(t), \cdots, P_{n}(t)\}$ inside the sphere and four points $\{A,B,C,D\}$ on the sphere}\label{triangulation_R3}
\end{figure}

On the other hand, as described in Fig.~\ref{pachner_move}, the moving of a vertex of the triangulation can be described by applying {\it the Pachner moves} to the triangulation of a 3 dimensional space.
\begin{figure}[h!]
 \centering
 \includegraphics[width = 10cm]{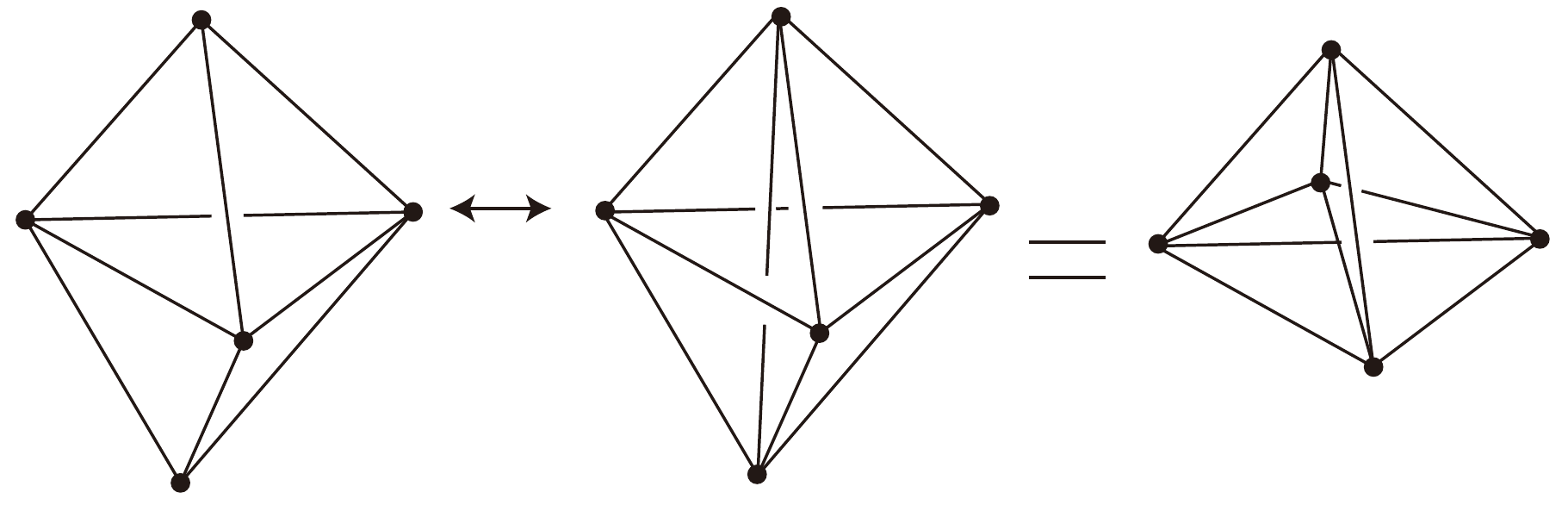}
 \caption{Pachner move and moving of a vertex of the triangulation of 3-dimensional space}\label{pachner_move}
\end{figure}
In other words, a path $\{P_{0}(t),\cdots, P_{n}(t)\}$ in $\pi_{1}(C_{n}'(\mathbb{R}^{3}))$ can be described by a finite sequence of ``Pachner moves'' applied to the triangulations of the sphere, see Fig.~\ref{point_moving_triangulation}.
\begin{figure}[h!]
 \centering
 \includegraphics[width = 10cm]{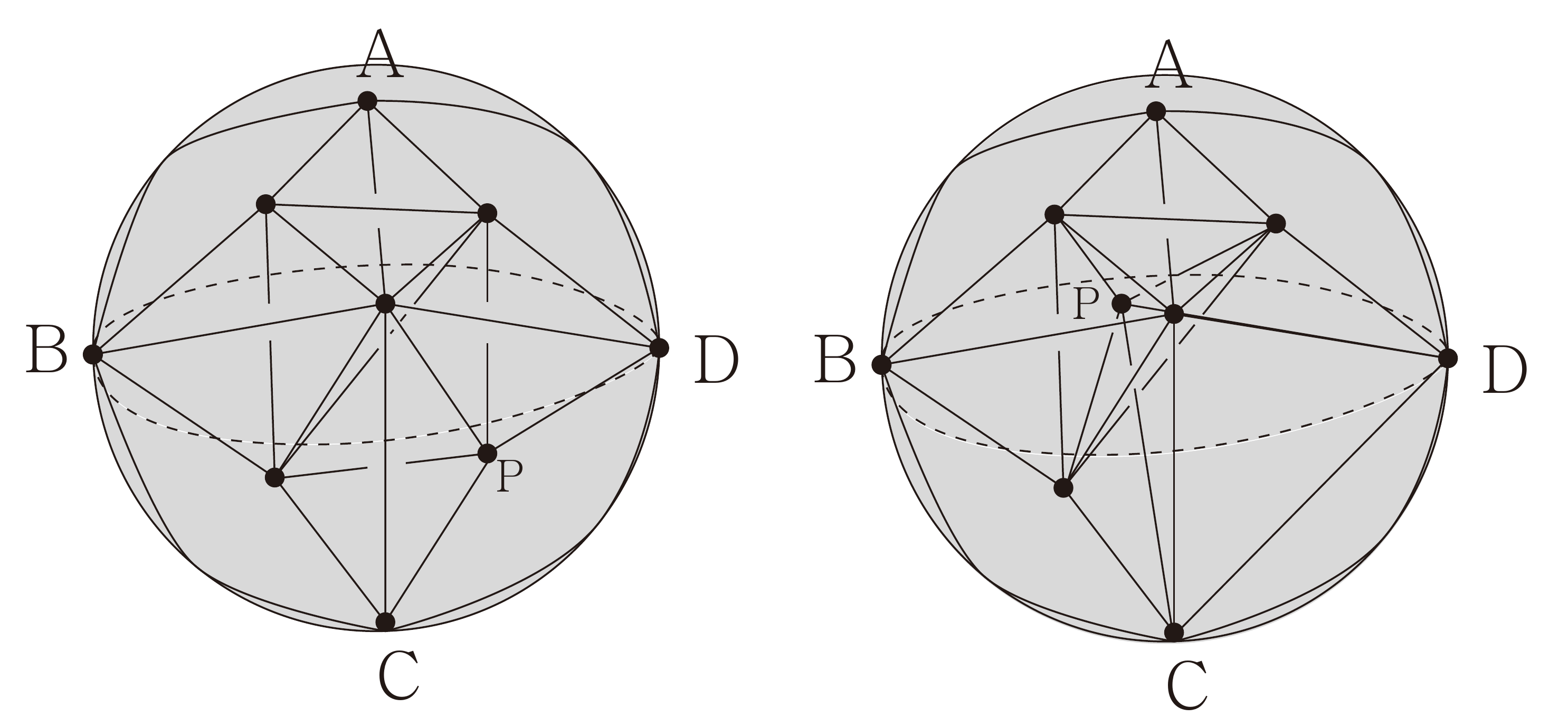}
 \caption{Applying a Pachner move to the triangulation of a 3 dimensional space and the moving of a vertex}\label{point_moving_triangulation}
\end{figure}

It can be expected that by using the triangulation of a sphere with $n+4$ points and the sequence of the Pachner moves, we obtain an invariant for (pure) braids.

\end{document}